\documentclass[12pt,twoside]{amsart}
\usepackage[margin=3cm]{geometry}
\usepackage[colorlinks=false]{hyperref}
\usepackage[english]{babel}
\usepackage{graphicx,titling}
\usepackage{float}
\usepackage{amsmath,amsfonts,amssymb,amsthm}
\usepackage{lipsum}
\usepackage[T1]{fontenc}
\usepackage{fourier}
\usepackage{color}
\usepackage[latin1]{inputenc}
\usepackage{esint}
\usepackage{caption}
\usepackage{ccicons}

\makeatletter
\def\blfootnote{\xdef\@thefnmark{}\@footnotetext}
\makeatother

\newcommand\ccnote{
    \blfootnote{\copyright\,\, Peter M. Topping and Hao Yin}
    \blfootnote{\ccLogo\, \ccAttribution\,\, Licensed under a \href{https://creativecommons.org/licenses/by/4.0/}{Creative Commons Attribution License (CC-BY)}.}
}

\usepackage[export]{adjustbox}
\numberwithin{equation}{section}
\usepackage{setspace}

\renewcommand{\leq}{\leqslant}

\renewcommand{\geq}{\geqslant}
\renewcommand{\mathbb}{\varmathbb}
\usepackage{fancyhdr}
\newtheorem{theorem}{Theorem}[section]
\newtheorem{lemma}[theorem]{Lemma}
\newtheorem{corollary}[theorem]{Corollary}
\newtheorem{proposition}[theorem]{Proposition}
\newtheorem{remark}[theorem]{Remark}
\usepackage{wasysym, paralist, latexsym, url}

\newcommand{\downto}{\downarrow}
\newcommand{\upto}{\uparrow}
\newcommand{\lap}{\Delta}
\newcommand{\weakto}{\rightharpoonup}
\newcommand{\vph}{\varphi}
\newcommand{\R}{\ensuremath{{\mathbb R}}}
\newcommand{\C}{\ensuremath{{\mathbb C}}}
\newcommand{\N}{\ensuremath{{\mathbb N}}}

\newcommand{\la}{\lambda}
\newcommand{\Om}{\Omega}
\newcommand{\ep}{\varepsilon}
\newcommand{\h}{\ensuremath{{\mathcal H}}}
\newcommand{\half}{\frac{1}{2}}
\newcommand{\de}{\delta}
\newcommand{\ga}{\gamma}
\newcommand{\ti}{\tilde}
\newcommand{\cb}{\ensuremath{{\mathcal B}}}
\newcommand{\si}{\sigma}
\newcommand{\abs}[1]{\left\vert#1\right\vert}
\newcommand{\set}[1]{\left\{#1\right\}}
\newcommand{\be}{\beta}
\newcommand{\bx}{{\bf x}}
\newcommand{\grad}{\nabla}
\newcommand{\al}{\alpha}

\newcommand{\pl}[2]{{\frac{\partial #1}{\partial #2}}}

\DeclareMathOperator{\Vol}{Vol}

\newtheorem{conj}[theorem]{Conjecture}
\newtheorem{ex}[theorem]{Example}

\address{Peter M. Topping, Mathematics Institute, University of Warwick, Coventry,
CV4 7AL, UK}
\email{p.m.topping@warwick.ac.uk}
\medskip
\address{Hao Yin, School of Mathematical Sciences, University of Science and Technology of China, Hefei, 230026, China}
\email{haoyin@ustc.edu.cn}

\begin{document}

\thispagestyle{empty}

\begin{minipage}{0.28\textwidth}
\begin{figure}[H]
%\centering
\includegraphics[width=2.5cm,height=2.5cm,left]{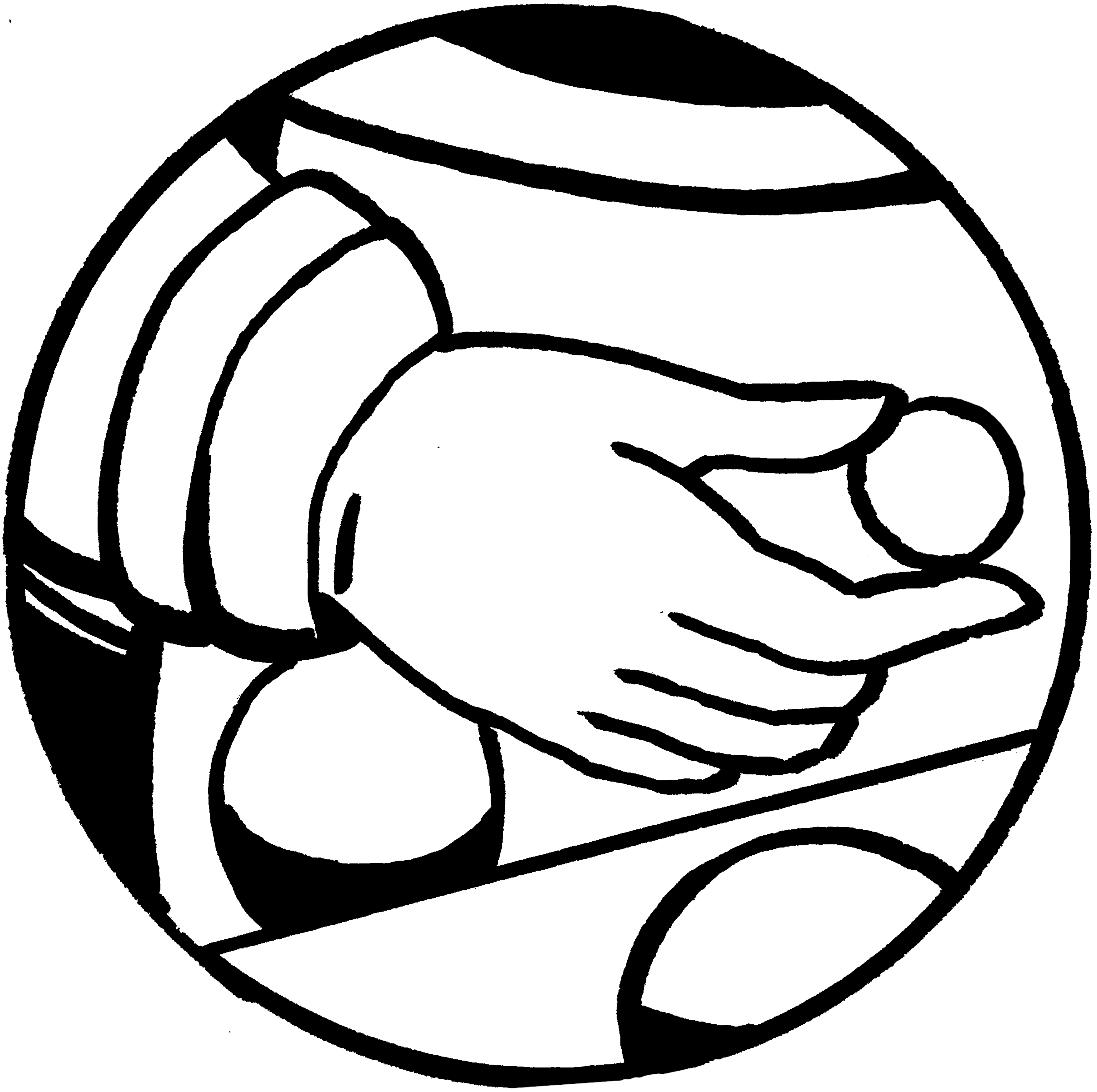}
\end{figure}
\end{minipage}
\begin{minipage}{0.7\textwidth} 
\begin{flushright}
%% The following metadata, in particular
%% the Paper No. and the DOI will be inserted by the journal
Ars Inveniendi Analytica (2024), Paper No. 4, 28 pp.
\\
DOI 10.15781/4x5c-9q97
\\
ISSN: 2769-8505
\end{flushright}
\end{minipage}

\ccnote

\vspace{1cm}

%%      -------------------------------------------------------------------------------
%%      -------------------------- TITLE ----------------------------
%%      -------------------------------------------------------------------------------
%% Authors, please put here the full title of the article

\begin{center}
\begin{huge}
\textit{Smoothing a measure on a \\ Riemann surface using Ricci flow}

%\textit{some titles take two lines}

\end{huge}
\end{center}

\vspace{1cm}

%%      -------------------------------------------------------------------------------
%%      -------------------------- AUTHORS AND AFFILIATIONS ----------------------------
%%      -------------------------------------------------------------------------------
%% Authors, please put here your full names and affiliations

%\hfill
\begin{minipage}[t]{.45\textwidth}
%\begin{minipage}[t]{.28\textwidth}
\begin{center}
{\large{\bf{Peter M. Topping}}} \\
\vskip0.15cm
\footnotesize{University of Warwick}
\end{center}
\end{minipage}
\hfill
\noindent
\begin{minipage}[t]{.5\textwidth}
%\begin{minipage}[t]{.28\textwidth}
\begin{center}
{\large{\bf{Hao Yin}}} \\
\vskip0.15cm
\footnotesize{University of Science and \\ Technology of China}
\end{center}
\end{minipage}
%\hfill
%\noindent
%\begin{minipage}[t]{.28\textwidth}
%\begin{center}
%{\large{\bf{Fran\c{c}ois Plum}}} \\
%\vskip0.15cm
%\footnotesize{Univerisity of Maryland} 
%\end{center}
%\end{minipage}

\vspace{1cm}

%%% Please replace "James Mustard" below 
%%% with the name of the managing editor for your submission.
%%% If you are unsure about their identity
%%% please ask an editor-in-chief about.

\begin{center}
\noindent \em{Communicated by Ben Andrews}
\end{center}
\vspace{1cm}

%%      -------------------------------------------------------------------------------
%%      -------------------------- BEGIN ABSTRACT ----------------------------
%%      -------------------------------------------------------------------------------
%% Authors, please put here the ABSTRACT and KEYBOARDS

\noindent \textbf{Abstract.} \textit{
We formulate and solve the existence problem for Ricci flow on a Riemann surface with initial data given by a Radon measure as volume measure.%conformal factor. 
The theory leads us to a large class of new examples of nongradient expanding Ricci solitons, including the first example of a nongradient K\"ahler Ricci soliton. It also settles the question of whether a smooth flow for positive time that attains smooth initial data in a distance metric sense must be smooth down to the initial time. We disprove this by giving an example of a complete Ricci flow starting with the Euclidean plane that is not the static solution.
}
\vskip0.3cm

\vspace{0.5cm}

%%      -------------------------------------------------------------------------------
%%      -------------------------- BEGIN ARTICLE ----------------------------
%%      -------------------------------------------------------------------------------
%% Authors, copy the body of your paper here

%\section{Section} 
%
%\subsection{Subsection} The main result of this paper is stated in the following theorem where we have set
%\[
%\partial\Omega=e^{-t^2/2}\,U
%\]
%where $U$ is sub-abelian ring in $\mathbb{S}^N$.
%
%\begin{theorem}\label{THM}
%If only $n$ would be an integer, then $\mathbb{R}^n$ would be a good old Euclidean space, with
%\[
%{\mathcal{E}}_f:=\int_{\mathbb{R}^n}f(x)\,d\mathcal{L}^nx
%\]
%well defined for every $f\in C^0_c(\mathbb{R}^n)$, in the notation of Federer.
%\end{theorem}
%
%Moreover, by Theorem~\ref{THM}
%\begin{equation}\label{EQ}
%\Delta\alpha=\gamma\le \beta.
%\end{equation}
%Exploting~\eqref{EQ}, it follows that
%\[
%\nabla\omega\cdot\omega\ge \delta^\pi.
%\]
%
%\begin{remark}\label{REM}
%{\rm Sometimes\footnote{This is a useless footnote to stress the importance of Remark~\ref{REM}.} is not always.}
%\end{remark}
%
%\begin{definition}
%We say\footnote{Yet another useless footnote, properly numbered.} that sometimes is always.
%\end{definition}
%
%For example we could notice that $\Omega\cap\partial[U_\eta(Z)]=A_0\cup A_1$ where all these sets are compact abelian super-rings in $\mathbb{R}^{n^2+33}_{\gamma_1,E,T_0}$.
%
%
%----------------------------------
%
%%%%%%%%%%%%%%%%%%%%%%%%%%%%%%%%%%%%%%%%%%%%%%%%%%%%%%

\section{Introduction}

Suppose we have a  Radon measure $\mu$ on a (connected, typically noncompact) Riemann surface $M$, and 
consider the naive question: 
\begin{quote}
\em Does there exist a unique smooth complete conformal Ricci flow $g(t)$ on $M$, for $t\in (0,T)$ (for some $T>0$), such that the Riemannian volume measure $\mu_{g(t)}$ converges weakly to $\mu$ as $t\downto 0$?
\end{quote}
By definition, the Ricci flow $g(t)$ is a solution of the equation $\pl{g}{t}=-2Kg$, 
where $K$ is the Gauss curvature, and thus preserves the conformal structure.
When we say \emph{conformal} Ricci flow we mean that this conformal structure agrees with the conformal structure of the Riemann surface.
If we write the flow with respect to a local complex coordinate $z=x+iy$ as 
$g(t)=u(t)(dx^2+dy^2)$ then the conformal factor $u$ satisfies the logarithmic fast diffusion 
equation $u_t=\lap\log u$.

%\cmt{Notation/presentation decision: write $|dz|^2$ or $(dx^2+dy^2)$?}

Here we say $\mu_{g(t)}$ converges weakly to $\mu$, and write 
$\mu_{g(t)}\weakto \mu$, if for all $\psi\in C_c^0(M)$ we have
$$\int_M \psi d\mu_{g(t)} \to \int_M \psi d\mu,$$
as $t\downto 0$. This is weak-* convergence when viewed in the dual of 
$C_c^0(M)$.

By virtue of the requested uniqueness, such a flow would give a natural smoothing of a measure while respecting the conformal structure. For example,
whenever we have a conformal automorphism $\vph:M\to M$, the smoothing of the pull-back measure
$\vph^*\mu:=(\vph^{-1})_\#\mu$ would coincide with the pull-back 
$\vph^* g(t)$ of the smoothing $g(t)$ of $\mu$.
%\cmt{intentionally suppressing discussion of existence time.}

%\cmt{
%$$\int_M\psi d\mu_{\vph^*g(t)}=\int_M\psi\circ\vph^{-1} d\mu_{g(t)}
%\to \int_M\psi\circ\vph^{-1} d\mu = \int_M\psi d((\vph^{-1})_\#\mu)
%=\int_M\psi d\vph^*\mu$$
%}

%\cmt{
%More generally,  if $\vph:N\to M$ is a conformal covering map, where $N$ is another Riemann surface, then it is possible to lift $\mu$ to a measure 
%$\vph^*\mu$ on $N$; then the smoothing of $\vph^*\mu$ is the pull back 
%$\vph^* g(t)$. In particular we can lift to the universal cover where the automorphism group may be larger.
%}

From this we can see that our initial question is too naive. If we could find a Ricci flow $g(t)$, $t\in (0,T)$, smoothing out a Dirac delta mass centred at the origin of $\R^2$, then $\vph^* g(t)$ would be another such smoothing solution for any dilation 
$\vph:\R^2\to\R^2$ defined by $\vph(x)=\la x$, for fixed $\la>0$. By uniqueness we would have to have $\vph^* g(t)=g(t)$, which is impossible for $\la\neq 1$.
This principle is applied in \cite{hui} where nonexistence is established by proving uniqueness for this specific initial data.

The problem here is not just the existence of a nontrivial conformal automorphism that leaves the measure invariant. For example, if we adjust the measure above to be the sum of the Dirac delta mass and Lebesgue measure on the whole of $\R^2$, then 
there cannot exist any onward solution achieving this weakly as $t\downto 0$.
The flow is unable to concentrate mass at an isolated point as $t\downto 0$.
More generally, as a preliminary result we will prove:
\begin{theorem}[No Dirac mass in initial data]
\label{no_dirac_mass_thm}
Suppose $M$ is any Riemann surface and $\mu$ is a Radon measure on $M$.
Suppose that for some open set $\Om\subset M$ the restriction
$\mu\llcorner \Om$ is the sum of a smooth measure (i.e. the volume measure of a smooth conformal, possibly degenerate, Riemannian metric on $\Om$) 
and a Dirac mass centred at a point within $\Om$.
Then there cannot exist $\ep>0$ and a smooth complete conformal Ricci flow $g(t)$ on $M$ for 
$t\in (0,\ep)$ such that $\mu_{g(t)}\weakto \mu$ as $t\downto 0$.
\end{theorem}
%
%\cmt{later, give fake proof..}
%\cmt{Keep in mind our wish to prove this by bounding it above by the instantaneously complete flow starting with the punctured plane.}
%
Despite these problems, our main result tells us that as soon as we rule out atoms in the initial measure we obtain existence.

%\cmt{Another way of saying nonatomic measure is `continuous measure'. We just give the explicit definition}
%
%\cmt{Below, the term nontrivial rules out measures that measure all sets as zero measure}

\begin{theorem}[Main existence theorem]
\label{main_thm}
Suppose $M$ is any (connected, possibly noncompact) Riemann surface and $\mu$ is any 
(nonnegative) nontrivial Radon measure on $M$  that is nonatomic in the sense that
$$\mu(\{x\})=0 \text{ for all }x\in M.$$
Define 
\begin{equation}
\label{T_def}
T:=\left\{
\begin{aligned}
\textstyle \frac{\mu(M)}{4\pi}\qquad & \text{if }M=\C\simeq \R^2\\
\textstyle \frac{\mu(M)}{8\pi}\qquad & \text{if }M=S^2\\
\infty\qquad & \text{otherwise}.
\end{aligned}
\right.
\end{equation}
Then there exists a smooth complete conformal Ricci flow $g(t)$ on $M$, for $t\in (0,T)$, such that the Riemannian volume measure $\mu_{g(t)}$ converges weakly to $\mu$ as $t\downto 0$.

In the cases that $T<\infty$, as $t\upto T$ we have 
$$\Vol_{g(t)}(M)=(1-{\textstyle \frac{t}{T}})\mu(M)\to 0.$$
Moreover, if $\mu$ has no singular part then $\mu_{g(t)}\to \mu$ in $L^1_{loc}(M)$.
More generally, if $\Om$ is the complement of the support of the singular part of $\mu$, 
then $\mu_{g(t)}\llcorner\Om \to \mu\llcorner\Om$ in $L^1_{loc}(\Om)$.
\end{theorem}

%\cmt{also, the conformal factor increases linearly locally}

%\cmt{That the volume will actually decrease linearly to zero is only following a posteriori from \cite{GT2}: In our limit construction we might lose mass at spatial infinity.}

To clarify the final assertion we need to recall that the Lebesgue-Radon-Nikodym
theorem implies that if we pick an arbitrary smooth conformal metric $g_0$ on $M$, with volume measure  $\mu_{g_0}$, then we can write $\mu$ as the sum of two mutually singular measures $\mu_{ac}$ and $\mu_{sing}$ where the absolutely continuous part
$\mu_{ac}$ can be written 
$\mu_{ac}=u_0\mu_{g_0}$ 
%$\mu_{ac}=\mu_{g_0}\llcorner u_0$ 
for some nonnegative
$u_0\in L^1_{loc}(M)$ and the singular part satisfies $\mu_{sing}\perp \mu_{g_0}$. 
%\cmt{change $\mu_{g_0}\llcorner u_0$ to $u_0\mu_{g_0}$? And $u_0$ is nonnegative.}
The final assertion says first that if 
$\mu=\mu_{ac}$
and we write $g(t)=u(t)g_0$, then $u(t)\to u_0$ in $L^1_{loc}(M)$.
If the singular part is nontrivial then we have the same convergence in $L^1_{loc}(\Om)$ (though $\Om$ could be empty).

We stress that in Theorem \ref{main_thm} we are not just given a smooth underlying surface and a Radon measure. It is essential that we are also given the underlying complex structure for the problem to be natural.

In addition to the existence assertion of Theorem \ref{main_thm}, we also speculate that the flow  should be unique:\footnote{Since this paper originally appeared in 
preprint form, Conjecture \ref{uniqueness_conj} has been proved in \cite{TY4}.}
\begin{conj}
\label{uniqueness_conj}
Suppose $M$ is any (connected, possibly noncompact) Riemann surface and $\mu$ is any 
(nonnegative) nontrivial Radon measure on $M$ such that 
$$\mu(\{x\})=0 \text{ for all }x\in M.$$
Suppose that $\ep>0$ and that $\tilde g(t)$ is a smooth complete conformal Ricci flow on $M$, for $t\in (0,\ep)$ such that the Riemannian volume measure 
$\mu_{\tilde g(t)}$ converges weakly to $\mu$ as $t\downto 0$.
We conjecture that if $T$ and $g(t)$ are as in Theorem \ref{main_thm}, then 
$\ep\leq T$ and $g(t)=\tilde g(t)$ for all $t\in (0,\ep)$.
\end{conj}
The potential noncompactness of $M$ is at the heart of the difficulty in proving this conjecture.

%\cmt{slightly annoying above that we write $\ep\leq T$ when $T$ could be infinity}

%\cmt{Prove uniqueness if there is no singular part?}

%\cmt{Below, it's not enough to say $\mu$ has no singular part. We also need to say that the initial data is attained in $L^1_{loc}$}

One case in which  the uniqueness does follow from 
prior work is when the measure has no singular part and is attained in $L^1_{loc}$. The following can be derived from
the theory in \cite{ICRF_UNIQ} (see Lemma \ref{even_stronger_again_lemma} below for the type of estimate involved) where uniqueness was established for smooth initial data.

\begin{theorem}[Uniqueness for initial data attained in $L^1$]
\label{L1_uniq}
Suppose $M$ is any (connected, possibly noncompact) Riemann surface and $\mu$ is any 
(nonnegative) nonsingular nontrivial Radon measure on $M$.
Pick any smooth conformal metric $g_0$ on $M$ and write $\mu=u_0\mu_{g_0}$ for some nonnegative $u_0\in L^1_{loc}(M)$.
If $g(t)=u(t)g_0$ and $\tilde g(t)=\tilde u(t)g_0$ are smooth complete Ricci flows on $M$, for
$t\in (0,\ep)$,
with the property that both $u(t)\to u_0$ and $\tilde u(t)\to u_0$ in $L^1_{loc}(M)$
as $t\downto 0$, then we must have $g(t)\equiv \tilde g(t)$ for all $t\in (0,\ep)$.
\end{theorem}
%\cmt{Remember to lift to the universal cover above!}

%\cmt{THIS IS PROBABLY A DISTRACTION. IT CAN WAIT UNTIL WE PROVE UNIQUENESS!
%Uniqueness would end up implying continuous dependence of the flow $g(t)$ on the initial data because it will follow from our theory that if $\mu_i$ is a sequence of (nontrivial) nonatomic Radon measures on $M$ that converge weakly to a (nontrivial) nonatomic measure $\mu$ on $M$, 
%then the $g(t)$ that we construct starting from $\mu$, existing on a time interval 
%$(0,T)$, will arise as a smooth local limit of the 
%flows $g_i(t)$ that we construct starting from $\mu_i$.}
%\cmt{In particular, the existence times will be sufficiently long.}

%\cmt{THIS CAN WAIT: Do we want to state a theorem stronger than the existence that gives uniform bounds on the conformal factor at time $t$, only in terms of the concentration profile of the initial measure? This might give a stronger statement about continuous dependence of solutions on the initial measure...}

Equipped with the well-posedness theory for measure-valued initial data that we have just outlined, we are able to generate a large class of new expanding Ricci solitons that 
include the first known examples of nongradient Ricci solitons in two dimensions, and the first known nongradient K\"ahler Ricci solitons;
see Section \ref{soliton_sect}. (The examples can be trivially extended to higher dimensions also.) 
This answers \cite[Problem 1.87]{chowIIgeometric}.
The simplest example that illustrates this idea is when we start the Ricci flow with a measure 
$2\pi \h^1\llcorner L$ concentrated on a line in $\R^2$, where $\h^1$ is one-dimensional Hausdorff measure, and $L$ is a line in the plane such as the $y$-axis. By the invariance of the initial data under scaling, and the existence theory in this paper, we can show that a Ricci flow emanates from this measure that is an expanding soliton, but not a gradient soliton.
In fact, in this case we observe that the soliton flow can be written down explicitly as $g(t)=u(x,y,t)(dx^2+dy^2)$ where
$$u(x,y,t)=\frac{2t}{t^2+x^2}.$$
Amongst the many other examples that can be constructed using this technique are expanding solitons for which the corresponding diffeomorphisms necessarily move not just by dilation but also by rotation. As $t$ decreases to zero, these diffeomorphisms make an unbounded number of full rotations.

\vskip 0.2cm

As a further application of our theory, 
%in Theorem \ref{counter_ex_thm} of Section \ref{DSS_Q} 
we will settle 
one of the natural uniqueness questions for Ricci flow.
%the question of whether one can have a smooth Ricci flow for $t>0$ that converges to a smooth initial metric as $t\downto 0$ in the weak sense of convergence of the Riemannian distance,   
%but does not converge smoothly. It turns out that this can be engineered by starting the Ricci flow  with the Radon measure arising from adding Lebesgue measure on the plane to the Hausdorff measure $\h^1$ restricted to a line. 
%
Suppose $g(t)$ is a smooth complete Ricci flow on a surface $M$ for $t\in (0,\ep)$
with the property that the distance function $d_{g(t)}:M\times M\to [0,\infty)$
converges locally uniformly to the distance function $d_{g_0}:M\times M\to [0,\infty)$
of a smooth complete Riemannian metric $g_0$ as $t\downto 0$. 
It is a question raised by A. Deruelle and T. Richard as to whether this implies that $g(t)$ converges smoothly locally to $g_0$ as $t\downto 0$; see
\cite{TR, loss_of_init_data, DSS}.
A. Deruelle, F. Schulze and M. Simon  \cite{DSS} show that the answer is yes if certain curvature hypotheses are added to $g(t)$.

In the following theorem we demonstrate that the answer is no in general.
\begin{theorem}
\label{counter_ex_thm}
There exists a Ricci flow $g(t)$ on $\R^2$, for $t>0$, with initial data given by 
the measure
$$\mu:=\mu_{g_0}+\h^1\llcorner L,$$ 
where $g_0$ is the Euclidean metric,
so that $d_{g(t)}\to d_{g_0}$  locally uniformly as $t\downto 0$. 
\end{theorem}
%
%If we could appeal to the uniqueness conjecture \ref{uniqueness_conj}, then we would have said \emph{the} Ricci flow with $\mu$ as initial data. 
%In practice it will be convenient to take very specific smooth approximations to $\mu$ in order to then generate $g(t)$ as in the proof of Theorem \ref{main_thm}.
%
One interpretation of this result is that we have nonuniqueness for the Ricci flow starting at the Euclidean plane if we ask for the initial data to be attained in the sense of convergence of the Riemannian distance. One has the static flow that remains as the Euclidean plane, but also the flow given in  Theorem \ref{counter_ex_thm}.
A slight modification of Theorem \ref{counter_ex_thm} answers Problem 1.2 from 
\cite{loss_of_init_data}.
We prove Theorem \ref{counter_ex_thm} in Section \ref{DSS_Q}.

\vskip 0.2cm

Theorem \ref{main_thm} generalises the existence theory of \cite{GT2} to the case of measure-valued initial data. For comparison and later application, we recall:
\begin{theorem}[\cite{GT2}]
\label{main_prior_thm}
Suppose $(M,g_0)$ is any (connected, possibly noncompact) orientable Riemannian surface and define 
\begin{equation}
\label{T_def2}
T:=\left\{
\begin{aligned}
\textstyle \frac{\Vol_{g_0}(M)}{4\pi}\qquad & \text{if }(M,g_0)\text{ is conformally }\C\\
\textstyle \frac{\Vol_{g_0}(M)}{8\pi}\qquad & \text{if }(M,g_0)\text{ is conformally }S^2\\
\infty\qquad & \text{otherwise}.
\end{aligned}
\right.
\end{equation}
Then there exists a smooth Ricci flow $g(t)$ on $M$, for $t\in [0,T)$, such that 
$g(0)=g_0$ and $g(t)$ is complete for all $t\in (0,T)$.
In the cases that $T<\infty$, we have 
$\Vol_{g(t)}(M)=(1-\frac{t}{T})\Vol_{g_0}(M)\to 0$ as $t\upto T$.
\end{theorem}
In this smooth case uniqueness was proved  in \cite{ICRF_UNIQ}, and is included in Theorem \ref{L1_uniq}.

\begin{remark}
\label{punctured_plane_rmk}
Let us illustrate Theorem \ref{main_prior_thm} with one example that will serve us later. Suppose $M$ is the punctured disc $B\setminus\{0\}$, and $g_0$ is the standard flat metric on $M$. This metric is incomplete because both the boundary of $B$ and the puncture can be reached from the interior in a finite distance. Theorem \ref{main_prior_thm} solves this by immediately blowing up near the boundary of $B$ to look like a Poincar\'e metric scaled to have curvature $-1/(2t)$, and blowing up near the puncture to look like a hyperbolic cusp scaled to have curvature $-1/(2t)$. It will be important later that a hyperbolic cusp has finite area.
\end{remark}

Theorem \ref{main_thm} also connects with a number of earlier works in which rough initial data is considered for Ricci flow. In one direction there has been much recent progress in understanding the existence problem for Ricci flows starting with metric spaces having suitably controlled geometry, e.g. \cite{Hochard, ST2, BCRW, TR}. 
There is a closer connection to earlier work in K\"ahler Ricci flow. Although the subtleties of this paper (and of \cite{GT2, ICRF_UNIQ}) largely stem from the possible noncompactness of the underlying space $M$, if we were to impose compactness then one could appeal to work of 
Guedj and Zeriahi \cite{Guedj} for existence and Di Nezza and Lu \cite{DNL} for uniqueness.
As we will see during the paper, some of the techniques we have at our disposal originate in the study of the logarithmic fast diffusion equation. See \cite{Vazquez} for an overview.

\begin{remark}[Recent developments]
\label{recent_dev_rmk}
%\emph{Recent developments:}
Since this work originally appeared in preprint form there have been further developments in the theory. In \cite{TY4} the uniqueness of Conjecture \ref{uniqueness_conj} was proved.
One immediate consequence is that the existence theory of this paper must always preserve symmetries of the initial data, and in \cite{PT} this has led to a complete classification of two-dimensional expanding Ricci solitons in two dimensions, following the ideas of this paper. 
The examples known prior to the present paper were all \emph{gradient} Ricci solitons, which turn out to represent a negligible proportion of all expanding Ricci solitons.

Given this existence and uniqueness theory, Lemma \ref{short_time_lem_updated} below
tells us that we have continuous dependence of solutions on the initial data, giving full well-posedness.

In \cite{PT} it was established that the type of initial data we introduce in this paper is the most general that could ever be considered for complete two-dimensional Ricci flows. This was achieved by showing that any complete two-dimensional Ricci flow on a time interval $(0,T)$ \emph{must} admit a weak limit $\mu$ as $t\downto 0$ that is a \emph{nonatomic} Radon measure. 
This gives insight into Theorem \ref{no_dirac_mass_thm}. More significantly, it 
implies that \textbf{all}  complete two-dimensional Ricci flows, $t\in (0,T)$, must arise from the main existence theorem \ref{main_thm} in this paper, 
once it has been trivially modified to handle non-orientable $M$ and trivial $\mu$.
This then completes the existence and uniqueness theory for any type of initial data in two dimensions.
\end{remark}

%\noindent
%\cmt{Additional points for later consideration:
%\begin{compactenum}
%\item
%Can we add conditions to a notion of smoothing of a measure so that it HAS to be Ricci flow? What is the story for Calabi flow?
%\end{compactenum}
%}

\vskip 0.2cm
\noindent
\emph{Acknowledgements:}
PT was supported by EPSRC grant  EP/T019824/1.
HY was supported by NSFC grant 11971451.
For the purpose of open access, the authors have applied a creative commons attribution (CC BY) licence to any author accepted manuscript version arising.

\section{Pointwise a priori bounds for Ricci flows on surfaces}
\label{upper_lower_bounds}

Before we start the proof of Theorem \ref{main_thm}, we recall and adapt some a priori upper and lower bounds for the conformal factors of Ricci flows on surfaces.

\subsection{Upper bounds for the conformal factors}

If we have a Ricci flow $g(t)$ on a Riemann surface, for $t\in [0,T)$, and we know that on some simply connected subdomain $B\subset M$ we have 
\begin{equation} \label{pw_ub}
g(0)\leq \la h,
\end{equation}
 where $h$ is the unique complete conformal hyperbolic metric on $B$, and $\la>0$, then this ordering is preserved under Ricci flow and we deduce that $g(t)\leq (\la+2t)h$ for all $t$.
In particular, we obtain a purely local ongoing upper bound for the flow. 

In this paper we need to deduce a uniform local upper bound without being able to assume 
an initial $L^\infty$ bound of the form \eqref{pw_ub}. We will only have access to initial $L^1$ control. 
The estimate we require is contained in the following theorem, which can be derived from 
Theorem 1.1 in \cite{TY1}.
%\cmt{Note that the $L^1$ bound on $u(0)$ is implicit here because otherwise the conclusion is vacuous.}

%\cmt{While looking for invariant ways of phrasing other estimates, I noticed that there is another invariant way of phrasing this estimate in terms of the conformal radius of an arbitrary simply connected domain. However, because that is equivalent to the radius of the largest inscribed circle, and because we have an unknown constant $C$ in the conclusion (currently) there does not appear to be any gain of information.}

\begin{theorem}
\label{TY1_thm}
Suppose $g(t)=u(t)(dx^2+dy^2)$ is a smooth conformal Ricci flow on the ball $B_{2r}$, for some $r>0$ and all $t\in [0,T)$, that need not be complete.
If $t\in (0,T)$ satisfies $t\geq \frac{\Vol_{g(0)}(B_{2r})}{2\pi}$ then 
$$\sup_{B_{r}} u(t)\leq C_0r^{-2}t,$$
where $C_0<\infty$ is universal.
\end{theorem}
It is worth emphasising that for smaller times, say for $t< \frac{\Vol_{g(0)}(B_{2r})}{4\pi}$, 
no such uniform upper bound on $u(t)$ can hold, see \cite{TY1}.

An intuitive way of viewing this result is that the Ricci flow averages the conformal factor of the initial metric over balls that are large enough to have initial volume of order $t$, at least in the sense of upper bounds.
This property can be compared with the smoothing of the normal linear heat equation, where averaging takes place over balls of \emph{radius} $\sqrt{t}$.

We will apply this theorem to obtain uniform upper bounds on a sequence of Ricci flows. %that approximate the one we are seeking in Theorem \ref{main_thm}. In addition, they will be a prerequisite in order to prove uniform lower bounds.

%\cmt{Sometimes I write $\C$ (when thinking in terms of Riemann surfaces) and sometimes I write $\R^2$. Probably we should just use one. Proposal: use $\R^2$, or write $\C\simeq\R^2$ where it helps.}

\begin{corollary}
\label{upper_cor}
Suppose $\Om\subset\R^2$ is open and $\mu$ is any nontrivial Radon measure on $\Om$ such that $\mu(\{x\})=0$ for all $x\in \Om$.
Suppose $g_i(t)=u_i(t)(dx^2+dy^2)$ is a sequence of smooth conformal Ricci flows (possibly incomplete) on $\Om$ over some common time interval $[0,T_1]$, with $T_1\in (0,\infty)$, such that the Riemannian volume measure 
$\mu_{g_i(0)}$ converges weakly to $\mu$.
Suppose that $0<\tau< T_1$ and $K\subset \Om$ is compact.
Then there exists a constant $L<\infty$ (depending on $\Om$, $\mu$, $T_1$, $\tau$, $K$) such that 
$$\sup_{K}u_i(t)\leq L$$
for all $t\in [\tau,T_1]$ and for all $i$ sufficiently large.
\end{corollary}

\begin{proof}
For each $x\in \Om$, we can pick $r>0$ sufficiently small so that $B_{3r}(x)\subset\Om$
and $\mu(B_{3r}(x))\leq \pi \tau$. Here we are using that $\mu$ is nonatomic.

By the weak convergence of $\mu_{g_i(0)}$ to $\mu$, we must have 
%\cmt{Evans-Gariepy page 65}
$$\limsup_{i\to\infty}\Vol_{g_i(0)}(\overline {B_{2r}(x)})\leq \mu(\overline {B_{2r}(x)}).$$
Therefore, for sufficiently large $i$ we  have $\Vol_{g_i(0)}({B_{2r}(x)})\leq 2\pi \tau$. 
%\cmt{Here we have to drop from $B_{3r}$ to $B_{2r}$ (say) because otherwise we may be losing mass at the boundary.}

As a result, for $i$ that large we are in a position to apply Theorem \ref{TY1_thm}
to give 
\begin{equation}
\label{local_upper}
\sup_{B_{r}(x)}u_{i}(t)\leq C_0 r^{-2}t \leq C_0r^{-2}T_1
\end{equation}
for all $t\in [\tau,T_1]$ and universal $C_0$. 
%\cmt{This $C_0$ now coincides with the one in Theorem \ref{TY1_thm}.}

We can make this analysis at each point $x\in K$, and then appeal to compactness of $K$ to reduce to finitely many such balls $B_{r_j}(x_j)$ 
%(with $r$ depending on $x$) 
that cover $K$. It remains to set $L$ to be the largest value of the right-hand side of \eqref{local_upper}, which will correspond to the ball of smallest radius.
\end{proof}

\subsection{Lower bounds for the conformal factors}

%\cmt{An alternative starting point for this section would be: Write
%$u/M_0=e^{-2v}$ so $v$ is a nonnegative solution to
%$$v_t=\frac{1}{u}\lap v.$$
%The upper bound for $u$ also gives us uniform ellipticity, but I'm not sure we need it.
%If we shift time forwards a little then the estimate $K\geq -1/(2t)$ will give us 
%a uniform lower Ricci bound, which then allows us to invoke Li-Yau's theory of differential Harnack estimates (I'd have to check that the time evolving metric did not cause problems). See Theorems 1.2, 2.1, 2.2 in their Acta paper.
%This is all fine, but it is unlikely to give the desired improvements to what we've already proven because the sharpness of the current estimates is unclear in the regime where 
%$u$ is very small, but this corresponds to $v$ being very large. The Moser or Li-Yau estimates are only likely to say that if $v$ is very large at time $\tau$ then it is at least $\ep$ times that at time $2\tau$, but this corresponds to the statement that 
%if $u$ is very small at time $\tau$ then it is as small at that to the power $\ep$ at time $2\tau$. Again the estimate is not linear.
%}

Lower bounds for solutions of the logarithmic fast diffusion equation are quite different in nature than upper bounds. 
It is no longer possible to derive local lower bounds given controlled initial data locally. Indeed, in \cite[Theorem A.3]{GT3} for arbitrarily small $\ep>0$ an example is given of a Ricci flow
$g(t)=u(t)(dx^2+dy^2)$ on the unit disc $B\subset\R^2$, for $t\in [0,\ep)$, such that $u(0)\equiv 1$ (i.e. we start with the flat unit disc) and yet $u(t)\to 0$ uniformly as $t\upto\ep$.

Loosely speaking, we will avoid such counterexamples by establishing that in the cases of interest if the conformal factor is too small at some point $({\bf x},t)$ in space time then a little later in time the conformal factor will also be small in a whole neighbourhood. This is then something that we can hope to rule out later by considering volume. Such an estimate can be viewed as a Harnack estimate for $-\log u$.
By considering an appropriately scaled version of the cigar soliton Ricci flow we find that an extra hypothesis is required to make this work, and here we assume an \emph{upper} bound for the flow. Note that an upper bound hypothesis $u\leq M_0$ will imply positivity of $-\log \frac{u}{M_0}$, as we expect to be required in order to prove a Harnack estimate.

%A first example to keep in mind is the cigar soliton: On $\R^2$ one can consider the flow $g(t)=u(t)(dx^2+dy^2)$ where
%$$u(t)=\frac{1}{e^{4t}+|{\bf x}|^2}.$$
%At every point, the conformal factor $u(t)$ decreases to zero.
%We can also consider rescalings of this flow, by pulling back by a dilation ${\bf x}\to \la {\bf x}$, to give 

In practice we will also appeal to a piece of global information, as established by B.L. Chen \cite{strong_uniqueness}, that 
on any complete $n$-dimensional Ricci flow the scalar curvature satisfies $R\geq -\frac{n}{2t}$. In particular, we have the following result 
that can be shown to be equivalent to the so-called Aronson-Benilan inequality
when working with maximal solutions on $\R^2$.
\begin{lemma}[\cite{strong_uniqueness}]
\label{chen_lower_lem}
If we have a complete Ricci flow $g(t)$ on a surface, $t\in (0,T)$, then $K\geq -\frac{1}{2t}$.
\end{lemma}
Two simple consequences of having a Ricci flow satisfying $K\geq -\frac{1}{2t}$
are that if we write $g(t)=u(dx^2+dy^2)$ in a local coordinate chart $B_{2\rho}(x_0)$,
so $K=-\frac{1}{2u}\lap \log u$, then 
\begin{equation}
\label{conseq1}
\lap \log u\leq \frac{u}{t}
\end{equation}
and 
\begin{equation}
\label{conseq2}
t\mapsto \frac{u(x,t)}{t} \quad\text{ is decreasing for each }x.
\end{equation}

In the following main result giving lower bounds we opt to assume the conclusion of the lemma above rather than assume completeness.

\begin{theorem}[cf. \cite{DDD} page 655]
\label{first_lower_bounds_thm}
Suppose we have a conformal Ricci flow $g(t)=u(t)(dx^2+dy^2)$ on a domain containing 
$\overline{B_{2\rho}(x_0)}$ for times in
$(0,t_2]$, and assume that $K\geq -1/(2t)$. 
Suppose in addition that $t_1\in (0,t_2]$ and that $u(x,t)\leq M_0$ for all $t\in [t_1,t_2]$ and $x\in B_{2\rho}(x_0)$.
Then for all $y_0\in B_\rho(x_0)$ we have
\begin{equation}
\label{conc_est}
\log \frac{M_0}{u(y_0,t_1)}\leq
4\log \frac{M_0}{u(x_0,t_2)}+4\log\frac{t_2}{t_1}+\frac{2\rho^2M_0}{(t_2-t_1)}
+\frac{\rho^2M_0}{8t_1}.
\end{equation}
\end{theorem}

Thus, if at the centre point $x_0$ at the later time $t_2$ we have a lower bound for $u$, then we get a lower bound at nearby points $y_0$ and earlier times $t_1$. 
This is allowing us to move the point $x_0$ to $y_0$ compared with the simple estimate that is following from the fact that $t\mapsto u(x,t)/t$ is decreasing.

%\cmt{When considering optimality of estimates, it seems to be  significant that we are only assuming the bound $u\leq M_0$ for times AFTER $t_1$. Before that time the flow has been regularising in the sense that $K\geq -1/(2t)$, but we could have had a decaying delta function, for example.}

%\cmt{We could formulate some concrete questions about optimality.
%For example: Is it possible to have a complete Ricci flow on $\R^2$ for $t\in (0,2]$,
%with $u\leq 1$ for $t\in [1,2]$, but with arbitrarily small $\frac{u(0,1)}{u(x,2)}$
%for some $x\in B$?
%Note that $\frac{u(0,1)}{u(0,2)}$ is bounded below by $\half$ because $\frac{u(0,t)}{t}$ is decreasing. Also, if we drop the hypothesis $u\leq 1$ then it fails  because we can have a very concentrated sphere centred at $x$.
%}

If we can establish a lower bound for $u(x_0,t_2)$ then 
Theorem \ref{first_lower_bounds_thm} gives \emph{some} lower bound for $u(y_0,t_1)$ but this bound is not very sharp as $t_1$ becomes small, and it relies on an upper bound
$u\leq M_0$ that may not hold for very small times.
However, if we use Theorem \ref{first_lower_bounds_thm} to give a lower bound for
$u(y_0,t_1)$ and then invoke the fact that $t\mapsto \frac{u(y_0,t)}{t}$ is 
decreasing (from \eqref{conseq2})
we obtain:
\begin{corollary}
\label{lower_cor}
In the setting of Theorem \ref{first_lower_bounds_thm}, 
for all $y_0\in B_\rho(x_0)$ and $t\in (0,t_1]$, we have
$$u(y_0,t)\geq \ep t$$
where $\ep>0$ can be given explicitly as
$$-\log\ep:=4\log \frac{M_0}{u(x_0,t_2)}+4\log t_2 - 3\log t_1-\log M_0
+\frac{2\rho^2M_0}{(t_2-t_1)}
+\frac{\rho^2M_0}{8t_1}.$$
\end{corollary}

The following proof  largely follows the work of Davis, DiBenedetto and Diller \cite{DDD}.

%\cmt{Concerning the global information discussed above, it is not actually the lower scalar bound that is ruling out the example of \cite[Theorem A.3]{GT3}.}

%\cmt{First part related to DDD section 3.2}

\begin{proof}[Proof of Theorem \ref{first_lower_bounds_thm}]
For $0<r\leq \rho$, define a function 
$$G(r,\rho):=\log \rho - \log r +\half\rho^{-2}(r^2-\rho^2).$$
We think of $\rho$ as a radius of a ball on which we will work, and $r$ as a polar radial coordinate on that ball. 
Observe the following properties:
\begin{compactenum}
\item
$G(\rho,\rho)=0$
\item
$\pl{G}{r}(\rho,\rho)=0$
\item
$0\leq G(r,\rho)\leq \log \rho - \log r$.
\end{compactenum}
Note above that the inequality $0\leq G$ follows because for $s\in (0,1]$ we can write
$$G(s\rho,\rho)=-\log s +\half(s^2-1)=:f(s),$$
and $f(s)\geq 0$ because $f'(s)\leq 0$ and $f(1)=0$.

For $z_0\in\R^2$ we are interested in the function on $B_\rho(z_0)$ defined by
$$x\mapsto G(|x-z_0|,\rho),$$
which inherits analogues of the properties above, and satisfies 
\begin{equation}
\label{G_int}
\int_{B_\rho(z_0)}G(|x-z_0|,\rho)dx=\int_0^\rho 2\pi r G(r,\rho)dr=
2\pi\rho^2 \int_0^1 s f(s) ds=\frac{\pi}{4}\rho^2
\end{equation}
and 
$$\lap G = -2\pi \de_{z_0}+2\rho^{-2}.$$
In particular, if $f\in C^2(\overline{B_\rho(z_0)})$ then 
\begin{equation}
\label{rep_form}
\fint_{B_\rho(z_0)} f=f(z_0)+\frac{1}{2\pi}\int_{B_\rho(z_0)}\lap f . G(|x-z_0|,\rho).
\end{equation}
Here, and in the sequel, we tend to omit $dx$ in integrals where we are integrating with respect to Lebesgue measure.
Applying this identity in the case that $f(x)=\log u(x,t)$, where 
$u$ is the conformal factor of a Ricci flow that lives on a domain containing 
$\overline{B_\rho(z_0)}$, gives (at some time $t$)
\begin{equation}
\label{rep_form2}
\fint_{B_\rho(z_0)} \log u=\log u(z_0)+\frac{1}{2\pi}\int_{B_\rho(z_0)}\lap \log u\, .\, G(|x-z_0|,\rho).
\end{equation}
%\cmt{cf. the mean value theorem for harmonic functions}
We will handle the $\lap\log u$ term in two different ways to give pointwise bounds from above and below. 
In the first way we appeal to \eqref{conseq1} and apply \eqref{rep_form2} with $z_0=y_0\in B_\rho(x_0)$ to give
$$\fint_{B_\rho(y_0)} \log u\leq \log u(y_0)+
\frac{1}{2\pi t}\int_{B_\rho(y_0)}u\, G(|x-y_0|,\rho).$$
Because  $u(t_1)\leq M_0$ on $B_{2\rho}(x_0)\supset B_\rho(y_0)$, equation \eqref{G_int} gives us
$$\fint_{B_\rho(y_0)} \log u(\cdot,t_1)\leq \log u(y_0,t_1)+ \frac{\rho^2M_0}{8 t_1},$$
i.e. 
\begin{equation}\begin{aligned}
\label{eq_one_way}
\log \frac{M_0}{u(y_0,t_1)}&\leq \left(\fint_{B_\rho(y_0)} \log \frac{M_0}{u(\cdot,t_1)}\right)+\frac{\rho^2M_0}{8 t_1}\\
&\leq
4\left(\fint_{B_{2\rho}(x_0)} \log \frac{M_0}{u(\cdot,t_1)}\right)+\frac{\rho^2M_0}{8 t_1}
\end{aligned}\end{equation}
This should be compared with (55) in \cite{DDD}.
Note that we again used the hypothesis $u\leq M_0$ to ensure that the integrand was nonnegative, allowing us to enlarge the domain of integration. (The factor $4$ comes from the fact that we are taking average integrals and we are quadrupling the area.)

There is another way we can make progress with \eqref{rep_form2}, again using
the inequality $K\geq -1/(2t)$, but this time using its consequence that $u(x,t)/t$ is decreasing in $t$, for each fixed $x$.
This will give an inequality in the other direction. 
%\cmt{equivalent to \cite{DDD} around page 655.}
%More precisely, we are going to use $K\geq -1/(2t)$ but in a different way.
%Combining with the Ricci flow equation gives
%$u_t=\lap\log u \leq u/t$, i.e. that $u(x,t)/t$ is decreasing in $t$, for each fixed $x$.
Take \eqref{rep_form2} at some time $s\in [t_1,t_2]$, 
%with $t_2\in (t_1,T)$, 
and $z_0=x_0$ and $\rho$ replaced by $2\rho$ 
%(strictly speaking we now need the Ricci flow chart defined a bit beyond $B_{2\rho}(x_0)$) 
(so now we need the Ricci flow to exist on $\overline{B_{2\rho}(x_0)}$).
Subtracting $\log s$ from both sides,  we obtain
$$\fint_{B_{2\rho}(x_0)} \log \frac{u(\cdot,s)}{s}
=\log \frac{u(x_0,s)}{s}+\frac{1}{2\pi}\int_{B_{2\rho}(x_0)} \lap\log u(x,s)\,G(|x-x_0|,2\rho),$$
and thus
$$\fint_{B_{2\rho}(x_0)} \log \frac{u(\cdot,t_1)}{t_1}
\geq\log \frac{u(x_0,t_2)}{t_2}+\frac{1}{2\pi}\int_{B_{2\rho}(x_0)} u_t(x,s)G(|x-x_0|,2\rho).$$
Averaging over $s$ (only the final term depends on $s$ now) gives
$$\fint_{B_{2\rho}(x_0)} \log \frac{u(\cdot,t_1)}{t_1}
\geq\log \frac{u(x_0,t_2)}{t_2}+\frac{1}{2\pi(t_2-t_1)}\int_{B_{2\rho}(x_0)} 
[u(x,t_2)-u(x,t_1)]G(|x-x_0|,2\rho).$$
Using $u(\cdot,t_2)\geq 0$ and $u(\cdot,t_1)\leq M_0$ again, as well as \eqref{G_int}, gives
\begin{equation}\begin{aligned}
\fint_{B_{2\rho}(x_0)} \log \frac{u(\cdot,t_1)}{t_1}
&\geq\log \frac{u(x_0,t_2)}{t_2}-\frac{M_0}{2\pi(t_2-t_1)}\int_{B_{2\rho}(x_0)} 
G(|x-x_0|,2\rho)\\
&=\log \frac{u(x_0,t_2)}{t_2}-\frac{\rho^2M_0}{2(t_2-t_1)}.
\end{aligned}\end{equation}
Rearranging gives
$$
\fint_{B_{2\rho}(x_0)} \log \frac{M_0}{u(\cdot,t_1)}
\leq
\log \frac{M_0}{u(x_0,t_2)}+\log\frac{t_2}{t_1}+\frac{\rho^2M_0}{2(t_2-t_1)}.
$$
This can now be combined with \eqref{eq_one_way} to conclude 
%\beq
%\label{conc_est}
$$\log \frac{M_0}{u(y_0,t_1)}\leq
4\log \frac{M_0}{u(x_0,t_2)}+4\log\frac{t_2}{t_1}+\frac{2\rho^2M_0}{(t_2-t_1)}
+\frac{\rho^2M_0}{8t_1}.
$$
%\end{equation}

\end{proof}

\section{{$L^1$ bounds for Ricci flows on surfaces}}

%We need to compile some basic estimates that prevent a local region from gaining or losing volume at an uncontrolled rate in the Ricci flows we consider.
In general, a locally defined Ricci flow can lose volume at uncontrolled rate. The following result rules this out when the Ricci flow in question lies within an instantaneously complete Ricci flow defined on the whole plane.\footnote{This lemma was subsequently refined in \cite[Lemma 2.2]{PT}.}

\begin{lemma}\label{lem:lowerL1bound_new}
Suppose $g(t)$ is an instantaneously complete Ricci flow on $\R^2$ for $t\in [0,T)$,
with $\Vol_{g(0)}(B)\geq v_0>0$.
%, where $T$ is the maximal existence time, i.e. $T=\infty$ if $\Vol_{g(0)}(\C)=\infty$ and $T=\Vol_{g(0)}(\C)/(4\pi)$ if $\Vol_{g(0)}(\C)<\infty$.
Then there exists $\eta>0$ universal such that for all $t\in [0,T)$ 
with $t<\eta v_0$ we have 
%Then there exists $T_0>0$ depending only on $v_0$ such that for all $t\in [0,T)$ with $t\leq T_0$, we have 
$$\Vol_{g(t)}(B_2)\geq v_0/2.$$
%	Let $B_1(\bx)\subset \Real^2$ and $u(t)$ be the maximal solution with smooth initial data satisfying $\int_{B_1(\bx)} u_0 >c>0$, then there is $t_0>0$ depending only on $c$ such that for $0<t<t_0$,
%\begin{equation*}
%	\int_{B_2(\bx)} u(t) >c/2.
%\end{equation*}
\end{lemma}
\begin{proof}
Let $\tilde{H}$ be the complete hyperbolic metric on $\R^2\setminus B_{3/2}$, and note that $\Vol_{\tilde H}(\R^2\setminus B_2)<\infty$.

Pick a sequence of smooth metrics $g_i$ on $\R^2$ such that $g_i\leq \tilde H/i$ (where defined) and so that $g_i\leq g(0)$ throughout $\R^2$, with $g_i\equiv g(0)$ on $B$.
Let $g_i(t)$ be the unique instantaneously complete Ricci flow on $\R^2$ starting with $g_i$. Then $g_i(t)\leq g(t)$, because $g(t)$ is `maximally stretched' \cite{GT2}, so $\Vol_{g(t)}(B_2)\geq \Vol_{g_i(t)}(B_2)$,
and $g_i(t)\leq (\frac{1}{i}+2t)\tilde H$, so 
$\Vol_{g_i(t)}(\R^2\setminus B_2)\leq (\frac{1}{i}+2t)\Vol_{\tilde H}(\R^2\setminus B_2)$.

Because $\Vol_{g_i}(\R^2)\geq \Vol_{g_i}(B)=\Vol_{g(0)}(B)\geq v_0$, we have
$\Vol_{g_i(t)}(\R^2)\geq v_0-4\pi t$ and $g_i(t)$ exists for at least a time $v_0/(4\pi)$.
Therefore 
\begin{equation}\begin{aligned}
\Vol_{g(t)}(B_2) &\geq \Vol_{g_i(t)}(B_2)\\
&=\Vol_{g_i(t)}(\R^2)-\Vol_{g_i(t)}(\R^2\setminus B_2)\\
&\geq v_0-4\pi t - (\textstyle \frac{1}{i}+2t)\Vol_{\tilde H}(\R^2\setminus B_2),
\end{aligned}\end{equation}
and by sending $i$ to infinity we deduce that
$$\Vol_{g(t)}(B_2) \geq 
v_0-t\left(4\pi + 2\Vol_{\tilde H}(\R^2\setminus B_2)\right),$$
which is enough to conclude.
\end{proof}

%\cmt{There is a similar lemma to the one above in the case that the Ricci flow lives on $S^2$. We use stereographic projection to map to $\R^2$. For simplicity, we do this so that $\mu(B)>0$. We still compare the subsequent Ricci flow to flows $g_i(t)$ that lie below $g(t)$ initially with equality on $B$, but now these are to exist on the whole of $S^2$.
%}

We will also require a bound in the opposite direction. In contrast to the estimate above, this is now a purely local assertion.

%\cmt{The original version has the wrong dependencies.}

\begin{lemma}\label{lem:upperL1bound_new}
Suppose $0<r<s$ and $g(t)$ is a conformal Ricci flow on $B_s$, for $t\in [0,T)$.
%with $\Vol_{g(0)}(B)<\infty$, 
Then there exists $\eta>0$ depending on $r/s$ such that
$$\Vol_{g(t)}(B_{r})\leq \eta t+ \Vol_{g(0)}(B_s)$$
for all $t\in [0,T)$.
\end{lemma}

This lemma can be proved directly via techniques developed for the study of the logarithmic fast diffusion equation. We will work via the following result, stated in 
\cite[Lemma 2.12]{TY2} based on the estimates in \cite{ICRF_UNIQ}.

\begin{lemma}
\label{even_stronger_again_lemma}
Suppose $1/2<r_0<r_0^{1/3}<R<1$, and $\gamma\in (0,\half)$.
Suppose $g_1(t)=u_1(t)(dx^2+dy^2)$ is any complete Ricci flow on $B$, 
and $g_2(t)=u_2(t)(dx^2+dy^2)$ is any Ricci flow 
on $\overline {B_R}$, both for $t\in (0,T]$. 
Then we have for all $t\in (0,T]$ that
\begin{equation}
\begin{aligned}
\left[\int_{B_{r_0}}(u_2(t)-u_1(t))_+ \right]^{\frac{1}{1+\ga}}
&\leq
\liminf_{s\downto 0}
\left[\int_{B_{R}}(u_2(s)-u_1(s))_+\right]^{\frac{1}{1+\ga}}\\
&\quad+C(\ga)
\left[\frac{t}{
(-\log r_0)\left[\log(-\log r_0)-\log(-\log R)\right]^\ga}
\right]^{\frac{1}{1+\ga}}.
\end{aligned}
\end{equation}
\end{lemma}

%\cmt{delete $d\bx$?}

\begin{proof}[{Proof of Lemma \ref{lem:upperL1bound_new}}]
By scaling the parametrisation, we may assume that $s=1$.
We apply Lemma \ref{even_stronger_again_lemma} with $g_1(t)$ the 
%(restriction of the) 
big-bang Ricci flow on $B$, i.e. $2t$ times the Poincar\'e metric $H=h(dx^2+dy^2)$, 
and $g_2(t)$ equal to the Ricci flow $g(t)=u(t)(dx^2+dy^2)$ we are analysing.
Evaluating the limit $s\downto 0$, we find that
\begin{equation*}
\begin{aligned}
\left[\int_{B_{r_0}}(u(t)-2th)_+\right]^{\frac{1}{1+\ga}}
&\leq
\left[\int_{B_{R}} u(0) \right]^{\frac{1}{1+\ga}}\\
&\quad+C(\ga)
\left[\frac{t}{
(-\log r_0)\left[\log(-\log r_0)-\log(-\log R)\right]^\ga}
\right]^{\frac{1}{1+\ga}}.
\end{aligned}
\end{equation*}
Taking the limit $R\upto 1$ and \emph{then} $r_0\upto 1$  gives 
\begin{equation*}
\int_{B}(u(t)-2th)_+ \leq \Vol_{g(0)}(B).
\end{equation*}
In particular, we have 
\begin{equation}\begin{aligned}
\Vol_{g(t)}(B_{r}) &\leq \Vol_{2tH}(B_{r})+ \Vol_{g(0)}(B)\\
&= 2t \Vol_{H}(B_{r})+ \Vol_{g(0)}(B)
\end{aligned}\end{equation}
for all $t\in [0,T)$, which is the lemma with an explicit value of $\eta$.
\end{proof}

\section{{The proof of the main existence theorem \ref{main_thm}}}

The proof will take the familiar structure that we will approximate the initial measure $\mu$ by objects that we know how to flow, and then try to take a limit of those flows.
In our case we approximate $\mu$ by the volume measure of smooth Riemannian surfaces, and then flow each one of these using the theory of instantaneously complete Ricci flows that was introduced in \cite{JEMS} and developed in \cite{GT2}. We need estimates that will  not only give convergence of a subsequence to a limit flow, but also guarantee that the initial data is not lost in the limit.

\subsection{The approximation}

It is standard to use mollification to smooth out a measure. In our situation that we have a Radon measure $\mu$ on a Riemann surface $M$, we can take a countable atlas on $M$ and use a partition of unity to decompose $\mu$ into a countable sum of measures that are each supported within a coordinate chart. We can then mollify within each chart to give  degenerate smooth conformal Riemannian metrics $u(dx^2+dy^2)$ and recombine a large but finite number of these to give a sequence of (typically degenerate) smooth conformal Riemannian metrics $\ti g_i$ on $M$ whose volume measures converge weakly to $\mu$ (and in $L^1_{loc}$ if $\mu$ has no singular part). If we now pick any  conformal Riemannian metric $g_0$ on $M$, then we can define smooth conformal \emph{nondegenerate} Riemannian metrics
$$g_i:=\ti g_i+\frac{g_0}{i}$$
whose volume measures will converge weakly to $\mu$ (and in $L^1_{loc}$ if $\mu$ has no singular part). 
It is a feature of such a construction that volume can only be lost, not gained, in the limit $i\to\infty$. By mollifying sufficiently aggressively and picking $g_0$ to have finite volume (which is possible because we are not insisting on completeness) we can ensure that the volumes $\Vol_{g_i}(M)$ are finite and converge to $\mu(M)$.  
To summarise, we find:

%\cmt{currently no nonatomic assumption nor finiteness of $\mu(M)$.}
\begin{lemma}
\label{smooth_approx_lem}
Let $\mu$ be a Radon measure on a Riemann surface $M$. Then there exists a sequence of smooth  conformal finite-volume Riemannian metrics $g_i$ on $M$ such that 
$$\mu_{g_i}\weakto \mu\qquad\text{ and }\qquad \Vol_{g_i}(M)\to \mu(M),$$
as $i\to\infty$.
If $\mu$ has no singular part, and thus if $g_0$ is any smooth conformal metric we can write 
$\mu=u_0\mu_{g_0}$ 
%$\mu=\mu_{g_0}\llcorner u_0$ 
where $u_0\in L^1_{loc}(M)$ is nonnegative,
then if we write $g_i=u_i g_0$ we may assume that $u_i\to u_0$ in $L^1_{loc}(M)$.
More generally, if $\Om$ is the complement of the support of the singular part of $\mu$
and we write $\mu\llcorner\Om=u_0\mu_{g_0}\llcorner \Om$ %\llcorner u_0$ 
%\cmt{NOTATION}
for nonnegative $u_0\in L^1_{loc}(\Om)$, then 
%if we write $g_i=u_i g_0$ 
we may assume that $u_i\to u_0$ in $L^1_{loc}(\Om)$.
\end{lemma}

%\cmt{
%$\Om$ is the union of all points that have some open neighbourhood that has zero measure for the singular part.
%}

%\cmt{Clarifications: $M$ is connected, maybe not closed, but no boundary is implicit.}

%\cmt{There were more details in the file here prior to 28.7.2021}

%\cmt{The approximating metrics need not be complete, but we could arrange that they are by making $g_0$ complete.}

\subsection{Generating the approximating Ricci flows}
\label{g_i_construct_sect}

Being smooth Riemannian metrics on a (connected) surface, each $g_i$ 
from Lemma \ref{smooth_approx_lem} can be evolved uniquely as an instantaneously complete Ricci flow $g_i(t)$ using Theorem \ref{main_prior_thm}.
The existence time interval is $[0,T_i)$ where
$T_i=\infty$ unless $M=S^2$,
in which case we are in the classical situation of Hamilton and Chow \cite{ham_surf, chow} and the existence time is $T_i=\frac{\Vol_{g_i}(M)}{8\pi}$, or 
$M=\C$ and $\Vol_{g_i}(M)<\infty$, in which case our solutions coincide with the maximal solutions of the logarithmic fast diffusion equation studied in \cite{DD}, for example,  and $T_i=\frac{\Vol_{g_i}(M)}{4\pi}$. 

%In both these exceptional cases the volume of $g_i(t)$ decreases linearly to zero as $t\upto T_i$.

By Lemma \ref{smooth_approx_lem}, we have 
$T_i\to T_\infty$, where $T_\infty:=\frac{\mu(M)}{8\pi}$ if $M=S^2$ and 
$T_\infty:=\frac{\mu(M)}{4\pi}$ if $M=\C$, but otherwise $T_\infty:=\infty$. 
%Moreover, if we succeed in proving that the Ricci flows $g_i(t)$ converge smoothly locally on $M\times (0,T_\infty)$ then the limit flow will have volume converging to zero as $t\upto T_\infty$ in both of these exceptional cases.
%\cmt{Not claiming that the volume decreases linearly because we'd have to prove that volume was not lost at spatial infinity in the limit $i\to\infty$.}

The essential point is that we can pick $T>0$ independent of $i$ such that $g_i(t)$ exists for all
$t\in [0,T)$.
%, for sufficiently large $i$. 
It is a further consequence of Theorem \ref{main_prior_thm}
that in the cases that $M=\C$ or $M=S^2$, for all $t\in (0,T)$ we have
\begin{equation}
\label{vol_bd}
\Vol_{g_i(t)}(M)\leq \Vol_{g_i(0)}(M)\to \mu(M)
\end{equation}
as $i\to\infty$. This will be useful to verify that the existence time of the Ricci flow $g(t)$ we are trying to construct cannot be larger than claimed.

%\cmt{What we are going to do is to verify that the approx flows converge over some small time interval, and have the right initial data. Having the right initial data implies that volume of the limit flow for infinitesimally small time is at least as large as $\mu(M)$, but we also want it to be no larger than $\mu(M)$ for $M=S^2,\C$.}

The key ingredient in order to construct $g(t)$, and to verify its initial condition, is to establish upper and lower bounds for the conformal factors of the approximate flows $g_i(t)$ that are uniform in $i$, using the estimates from Section \ref{upper_lower_bounds}. Parabolic regularity theory will then give us compactness.

%The remainder of the proof of Theorem \ref{main_thm} involves proving that the Ricci flows $g_i(t)$ converge smoothly locally on $M\times (0,T_\infty)$. This will follow if when we write our flows $g_i(t)$ in an arbitrary 

\subsection{Extracting a short time limit flow}

In Section \ref{g_i_construct_sect} we found conformal instantaneously complete Ricci flows $g_i(t)$ on $M$, for $t\in [0,T)$, with the volume measures of $g_i(0)$ converging weakly to $\mu$ as $i\to\infty$. The goal of this section is to prove the following compactness result for this or any other such sequence $g_i(t)$.
\begin{lemma}
\label{short_time_lem}
Suppose that we have a sequence of smooth instantaneously complete conformal Ricci flows $g_i(t)$ on a Riemann surface $M$, for $t\in [0,T)$, such that the volume 
measures of $g_i(0)$ converge weakly to a 
nontrivial
nonatomic Radon measure $\mu$.
Then after passing to a subsequence, and possibly reducing $T>0$, we have
$g_i(t)\to g(t)$ smoothly locally on $M\times (0,T)$, where $g(t)$ is itself a smooth complete Ricci flow on $M$ for $t\in (0,T)$. 
Moreover, the volume measures of $g(t)$ converge weakly to $\mu$ as $t\downto 0$.
\end{lemma}
%%
%\cmt{added COMPLETENESS OF $g(t)$ CONCLUSION}

The proof will require uniform ($i$-independent) a priori estimates on the conformal factors of the flows $g_i(t)$, which parabolic regularity will turn into $C^k$ estimates thus giving us the required compactness. The a priori estimates will also allow us to deduce completeness of the limit. 
The following is the main sub-lemma that gives these required estimates.
%will be required in the proof of Lemma \ref{short_time_lem}.
%

%\cmt{The exposition below has changed to cover a general choice of $g_i(t)$ rather than the one that we have explicitly constructed. Why is this important? Because it gives us a priori estimates for general solutions that may be useful in the uniqueness. More precisely, if $g(t)$ is some other Ricci flow with the correct measure initial data 
%then we can define $g_i(t)=g(t+\frac{1}{i})$, get the estimates for $g_i(t)$ and then take $i\to\infty$ to get estimates on $g(t)$ that don't depend on the particular choice of $g(t)$ we took. They hold for ANY solution $g(t)$.}
%\cmt{In particular, by this discussion near each point we get a lower bound for $u$ that controls the rate of decrease of the potential by the integrable function $\log t$. Could be useful in uniqueness.}
%\cmt{In the following lemma we use $\cb$ instead of $B$ for the coordinate neighbourhood so as to leave $B$ to describe the whole space in the case that $M$ is conformally the disc.}
\begin{lemma}
\label{short_time_lem_sub}
Under the hypotheses of Lemma \ref{short_time_lem} there exists $T_0\in (0,T)$ so that for every $\tau\in (0,T_0)$, 
each point $x\in M$ admits a coordinate neighbourhood $\cb$ such that after passing to a subsequence in $i$ and writing
$g_i(t)=u_i(t)(dx^2+dy^2)$ we have the uniform estimates
\begin{equation}
\label{required_upper}
u_i(t)\leq L \qquad\text{ on } \cb\times [\tau,T_0]
\end{equation}
and 
\begin{equation}
\label{required_lower}
u_i(t)\geq \ep t \qquad\text{ on } \cb\times [0,T_0]
\end{equation}
where $L$ and $\ep$ are independent of $i$ and $t$. 
In the special case that $M=B$, so we can define the conformal factors $u_i(t)$ globally in $B$, we have the improved estimate
\begin{equation}
\label{required_lower_B}
u_i(t)\geq 2th \qquad\text{ on } B\times [0,T)
\end{equation}
where $h(dx^2+dy^2)$ is the Poincar\'e metric on $B$.
%In the special case that $M=\R^2$, so we can define the conformal factors $u_i(t)$ globally in $\R^2$, we have the improved estimate for each $t\in [0,T_0]$ that
%***
%where $\eta>0$ is allowed to depend on $t$ (but not $i$) and $r^2:=x^2+y^2$.
\end{lemma}
%\cmt{above, $L$ and $\ep$ can depend on $x$ etc. and $L$ depends on $\tau$!}

It is not necessary to pass to a subsequence in the lemma above, but allowing this will give a small shortcut in the proof. 

There is one additional estimate giving a lower bound for the conformal factors that we will need in the case that $M=\R^2$ in order to show that the extracted limit is complete. 
\begin{lemma}
\label{R2_lower_lem}
If $e^{2v}(dx^2+dy^2)$ is a smooth complete Riemannian metric on $\R^2$ with Gauss curvature bounded below by $K\geq -\si_1$ throughout, and so that 
$\|v\|_{C^2(B_2)}\leq \si_2$, then there exists $\de>0$ depending only on $\si_1$ and $\si_2$ such that 
\begin{equation}
\label{required_lower_R2}
e^{2v}\geq \frac{\de}{(r\log r)^2} \qquad\text{ on }\ \R^2\setminus B_2
\end{equation}
where $r^2=x^2+y^2$.
\end{lemma}
\begin{proof}[Proof of Lemma \ref{R2_lower_lem}]
This result is very similar to \cite[Theorem 3.2]{GT2}, so we only sketch the proof.
Define a function on $\R^2\setminus B$ by 
$$\ga(x,y)=-\log (r\log r)$$
so that $e^{2\ga}(dx^2+dy^2)$ is the unique complete conformal hyperbolic metric on 
$\R^2\setminus B$. Choose a cut-off function $\vph\in C_c^\infty(B_2,[0,1])$ with
$\vph\equiv 1$ on $B_{3/2}$. Then we can define a complete interpolated metric 
$G:=e^{2w}(dx^2+dy^2)$ on $\R^2\setminus B$ by setting
$$w=\vph \ga + (1-\vph)v,$$
and it is easy to verify that the Gauss curvature $-e^{-2w}\lap w$ of $G$ is bounded below by some constant $-1/\de<0$ depending only on $\si_1$, $\si_2$ and our choice of $\vph$.
But the Gauss curvature of $\de e^{2\ga}(dx^2+dy^2)$ is identically $-1/\de$,
and so Yau's version of the Schwarz lemma, in the form given in \cite[Theorem B.3]{GT2}, tells us that $\de e^{2\ga}\leq e^{2w}$.
Restricting to $\R^2\setminus B_2$, where $w=v$, then gives
$$e^{2v}\geq \de e^{2\ga}=\frac{\de}{(r\log r)^2}.$$
\end{proof}

Let us assume Lemma \ref{short_time_lem_sub} for the moment and try to complete the proof of Lemma \ref{short_time_lem}.
%\cmt{Apply to the specific $g_i(t)$ flows from before...}

Let $T_0$ be as in Lemma \ref{short_time_lem_sub}. That lemma then tells us that for each $\tau\in (0,T_0)$, each point $x\in M$ admits a coordinate neighbourhood $\cb$ 
%where estimates \eqref{required_upper} and \eqref{required_lower} hold. 
%Lemma \ref{short_time_lem_sub} gives 
so that $u_i(t)$ is sandwiched between $\ep\tau$ and $L $ on $\cb\times [\tau,T_0]$. Such uniform control is enough to bring parabolic regularity theory into play and give uniform interior $C^k$ bounds on $\cb\times (\tau,T_0]$. Since $\tau\in (0,T_0)$ was arbitrary we deduce that a subsequence of $u_i(t)$ converges smoothly locally on $\cb\times (0,T_0]$ to the conformal factor $u(t)$ of a smooth Ricci flow for $t\in (0,T_0]$. 
By repeating for countably many suitable centre points $x\in M$ and taking a diagonal subsequence we deduce that there exists a smooth limit Ricci flow $g(t)$ on $M$ for $t\in (0,T_0]$, as required.

At this point we have to verify that $g(t)$ is complete. For this purpose, by lifting to the universal cover, we may as well assume that $M$ is either $S^2$, $B$ or $\R^2$, and the completeness is obvious in the case $M=S^2$.
If $M=B$, then the lower bound \eqref{required_lower_B} of Lemma \ref{short_time_lem_sub}
passes to the limit to give $u(t)\geq 2th$, and so $g(t)$ inherits the completeness of the Poincar\'e metric.
If $M=\R^2$, then we will use Lemma \ref{R2_lower_lem}.
For each $t\in (0,T_0)$ we have $K_{g_i(t)}\geq -1/(2t)$ throughout by Lemma \ref{chen_lower_lem}. If we write $g_i(t)=u_i(t)(dx^2+dy^2)$ globally on $\R^2$ then
we have already established $i$-independent  $C^2$ bounds on $B_2$, say, at each $t\in (0,T_0)$. (These bounds depend on $t$.)
Therefore Lemma \ref{R2_lower_lem} gives us a lower bound
\begin{equation}
u_i(t)\geq \frac{\de}{(r\log r)^2} \qquad\text{ on } \R^2\setminus B_2
\end{equation}
for positive $\de$ that can depend on $t$ etc. but not on $i$.
This bound passes to the limit, ensuring that the limit conformal factor does not decay so fast that completeness could fail.
This completes the proof that $g(t)$ is instantaneously complete in all cases.

In order to verify the weak convergence of $\mu_{g(t)}$ to $\mu$ it suffices to check it in some neighbourhood of an arbitrary point $x_0\in M$. This is provided by the following lemma.

%\cmt{this reduction to working in one chart uses a partition of unity. We need it anyway to lift to the universal cover - a general cut-off on $M$ lifts to something that no longer has compact support...}

%\cmt{No completeness assumed here}

\begin{lemma}
Suppose we have a sequence of Ricci flows $g_i(t)=u_i(t)(dx^2+dy^2)$ on $B$, each for 
$t\in [0,T)$, with the properties that $g_i(t)$ converges smoothly locally on $B\times (0,T)$ to a Ricci flow $g(t)=u(t)(dx^2+dy^2)$. 
Suppose that $\mu_{g_i(0)}\weakto \mu$ as $i\to\infty$ and we have the uniform control 
that there exists $\ep>0$ such that $u_i(t)\geq \ep t$ for all $i$ and all $t\in [0,T)$.
Then 
$$\mu_{g(t)}\weakto \mu\quad\text{ as } t\downto 0.$$
\end{lemma}

This lemma is pinning down a precise condition that allows us to interchange the limits
$t\downto 0$ and $i\to\infty$.

\begin{proof}
We need to establish that for all $\vph\in C_c^0(B)$ we have
$$\lim_{t\downto 0} \int_B \varphi u(t) = \int_B \varphi d\mu.$$
By approximation, we may assume that $\vph\in C_c^\infty(B)$.
What we know already is that 
$$\lim_{i\to \infty} \int_B \varphi u_i(0) = \int_B \varphi d\mu$$
and that for all $t\in (0,T)$,
$$\lim_{i\to \infty} \int_{B} \varphi u_i(t) = \int_{B} \varphi u(t).$$
It remains to prove
\begin{equation*}
\lim_{t\downto 0} \int_{B} \varphi (u_i (t) -u_i(0))=0
\end{equation*}
{\bf uniformly} with respect to $i$.

Recall that $u_i$ is a smooth solution to the equation $\partial_t u_i = \triangle \log u_i$. Pick $r\in (0,1)$ such that $B_r$ contains the support of $\vph$, 
and set $s=(1+r)/2$ so that $r<s<1$.
We have
\begin{eqnarray}\label{eqn:uniform}
	\abs{\partial_t \int_{B} \varphi u_i(t)} &=& \abs{\int_{B} \varphi \triangle \log u_i} = \abs{\int_{B} \triangle\varphi  \log u_i} 	\\ \nonumber
	&\leq& C \int_{B_r} \abs{\log u_i} \\ \nonumber
	&\leq&  C \int_{B_r\cap \set{u_i\geq 1}} u_i + C \int_{B_r\cap \set{u_i<1}} \abs{\log u_i} \\ \nonumber
	&\leq& \Vol_{g_i(t)}(B_r)+ C \abs{\log (\ep t)} \\ \nonumber
	&\leq& \eta t+\Vol_{g_i(0)}(B_s)+ C \abs{\log (\ep t)} \\ \nonumber
	&\leq& \eta t+\mu(\overline{B_s})+1+ C \abs{\log (\ep t)},
%	&\leq& C(T,\bx)  + C \abs{\log t}.
\end{eqnarray}
for sufficiently large $i$, using Lemma \ref{lem:upperL1bound_new}.
%Here we used the lower bound of $u_i$ and the fact that
%\begin{equation*}
%\int_{B_1} u_i \leq C
%\end{equation*}
%for $t<1$, which is a consequence of Lemma \ref{lem:upperL1bound}. The required uniform convergence follows from (\ref{eqn:uniform}) trivially.
Now that the right-hand side does not depend on $i$, this inequality can  be integrated with respect to $t$ to give the required control.
\end{proof}

\noindent
In order to prove Lemma \ref{short_time_lem} it thus suffices to prove Lemma \ref{short_time_lem_sub}.

\begin{proof}[{Proof of Lemma \ref{short_time_lem_sub}}]
By lifting to the universal cover it suffices to assume that $M$ is $S^2$, $\C$ or $B$.

\vskip 0.2cm
\noindent
{\bf The case $M=B$.}

\noindent
This is the easiest case, given the existing literature. We are free to pick  \emph{any} $T_0\in (0,T)$ (although $L $ will depend on $T_0$).
%To extract a limit Ricci flow $g(t)$, it suffices to show that for arbitrary $x_0\in M$ and arbitrary $t_0\in (0,T)$, there exists a conformal chart $\phi:B\to M$  sending $0$ to $x_0$ such that if we write 
%$\phi^*(g_i(t))=u_{i}(t)(dx^2+dy^2)$ on $B$, then we have an $i$-independent bound for $|\log u_i(t)|$ on $B_r\times [t_0,T)$, for some $r\in (0,1]$ and sufficiently large $i$. Parabolic regularity theory would then allow us to extract a subsequential limit in the interior of $B_r\times [t_0,T)$.
%To avoid notational confusion 
We will take our coordinate neighbourhood to be explicitly some ball $\cb=B_r(x)\subset B$.

For our given $x\in B$, pick $r>0$ sufficiently small so that $B_{2r}(x)\subset B$.
%We start by picking a conformal coordinate chart $B_2\subset D$. 
Corollary \ref{upper_cor}, applied with $K$ there equal to $\overline{B_r(x)}$ here, and $T_1$ there equal to $T_0$ here,
immediately gives us the required upper bound \eqref{required_upper}.

For the uniform lower bound \eqref{required_lower_B} -- and hence
\eqref{required_lower} -- on $u_i(t)$, consider the Poincar\'e metric 
$H=h(dx^2+dy^2)$ on $B$.
According to \cite{GT2}, the Schwarz-Pick-Ahlfors-Yau lemma can be applied in order to prove that any complete Ricci flow on $B$ lies above the \emph{big-bang} Ricci flow $2tH$, and in particular $g_i(t)\geq 2tH$ and hence $u_i(t)\geq 2th$  for all 
$t\in (0,T)$.
This completes the argument in the case $M=B$.

The remaining two cases are more subtle because the Schwarz-Pick-Ahlfors-Yau lemma cannot be applied to obtain the global information required for a uniform lower bound.

%\cmt{It would arguably be easier below if we dilated the original parametrisation (after fixing $v_0$) so that we have $v_0\leq \half \mu(B)$ but also obtain that $x\in B$. But I quite like having the $\rho$ given as it is - easier to see what is going on.}

\vskip 0.2cm
\noindent
{\bf The case $M=\C$.}

\noindent
By translating the solution within $\C$, we may assume that $\mu(B_{1/2})>0$.
If we set $v_0=\half \mu(B_{1/2})$ then $\Vol_{g_i}(B_{1/2})\geq v_0$ for sufficiently large
$i$. We can safely ignore the finitely many terms that fail this inequality. 
%\cmt{Evans-Gariepy p. 65}

Then Lemma \ref{lem:lowerL1bound_new} tells us that 
$\Vol_{g_i(t)}(B)\geq v_0/2$ provided $t\leq \eta v_0$.
(If necessary, reduce $v_0$ so that $\eta v_0<T$.)
We are already in a position to define $T_0:=\half \eta v_0$.
%(Without loss of generality, we may assume $2T_0<T$.)
Because of the volume bound at time $\eta v_0=2T_0$, 
for each $i$ there exists $x_i\in B$
such that 
$$u_i(x_i, 2T_0)\geq v_0/(2\pi).$$
The upper bound \eqref{required_upper} near $x$, e.g. on $\cb:=B_1(x)$ will follow 
immediately from Corollary \ref{upper_cor}. In fact, we will need a more extensive upper bound in order to obtain the required lower bound. With this in mind we set $\rho=|x|+2$ and apply Corollary \ref{upper_cor} with $T_1$ there equal to $2T_0$ here, $\tau$ there equal to $T_0$ here, and $K$ there equal to $\overline{B_{2\rho+1}(0)}$. 
We deduce a uniform upper bound that we write 
\begin{equation}
\label{upper_bd_T02T0}
u_i(t)\leq M_0\qquad\text{throughout }\overline{B_{2\rho+1}(0)}\times [T_0,2T_0].
\end{equation}

%We will need uniform upper bounds for $u_i$ between times $T_0$ and $2T_0$ on a suitably large ball. If $x\in \C$ is the point considered in the lemma, then we define 
%$\rho=|x|+2$ (soon this will be the $\rho$ in Corollary \ref{lower_cor}) and ask for a bound on the ball $B_{2\rho+1}(0)$.

We apply the Harnack inequality implication of Corollary \ref{lower_cor} with $x_0$ there equal to $x_i$ here, and with $[t_1,t_2]$ there equal to $[T_0,2T_0]$ here. 
Note that 
$\overline{B_{2\rho}(x_i)}\subset \overline{B_{2\rho+1}(0)}$ where we have the upper bound \eqref{upper_bd_T02T0} for $u_i(t)$.

We deduce lower bounds for $u_i$ at all points $y_0\in B_\rho(x_i)$.
But $B_\rho(x_i)\supset B_{|x|+1}(0) \supset B_1(x)$, and so we deduce the required lower bounds throughout $\cb=B_1(x)$, and at all times $t\in (0,T_0]$.

%For the point $x\in \C$ considered in the lemma, we would like to establish lower bounds at a general point $y_0\in B_1(x)$. 
%If we set $\rho=|x|+2$ then we can estimate
%$$|y_0-x_i|\leq |y_0-x|+|x-0|+|0-x_i|\leq 1+|x|+1=\rho,$$
%and so $y_0\in B_\rho(x_i)$.

%\cmt{could do same for $S^2$ - use M\"obius to get everything in one spot....}

\vskip 0.2cm
\noindent
{\bf The case $M=S^2$.}

\noindent
There are various ways that this case can be handled, exploiting the compactness of $S^2$. Indeed, we could have taken an approach as we did to handle the case $M=\C$ and not need to pass to a subsequence. Instead we take a shorter route.

For sufficiently large $i$ we must have $\Vol_{g_i(0)}(S^2)\geq \half \mu(S^2)$.
Because Ricci flows on $S^2$ lose area at a rate $8\pi$, we must then have
$\Vol_{g_i(t)}(S^2)\geq \half \mu(S^2)-8\pi t$.
In particular, if we fix $T_0:=\frac{1}{64\pi}\mu(S^2)$, then over the time interval $[0,2T_0]$ we have $\Vol_{g_i(t)}(S^2)\geq \frac14 \mu(S^2)$.
%Intuitively, this lower integral bound will allow us to pick a point in $S^2$ for each $i$ where the conformal factor of $g_i(2T_0)$ with respect to any fixed background metric is controlled from below. 
Because of the flexibility of M\"obius transformations, after passing to a subsequence and writing $g_i(t)=u_i(t)(dx^2+dy^2)$ with respect to an appropriate conformal coordinate chart $\R^2\subset S^2$, we can be sure both that there exists a sequence $x_i\to 0\in\R^2$ with 
$$u_i(x_i,2T_0)\geq \de>0,$$
for all $i$, and also that the point $x$ from the lemma lies in $B$. 

%\cmt{We could think about how to improve the exposition above. Here is one option. Take a round standard metric on $S^2$. Consider the conformal factor of $u_i(2T_0)$ relative to this. By the lower volume bound, for each $i$ we can find a point where this is at least $\mu(S^2)/(16\pi)$. Pass to a subsequence so that this sequence of points converges to some point $\tilde x$. Take one M\"obius transformation so that $x$ is strictly less that $\pi/2$ from this limit point. This changes the lower bound for the conformal factors, but we still have a positive  $i$-independent lower bound. Now take a stereographic chart about $\tilde x$. Thus $\tilde x$ becomes the origin, and $x$ ends up within the unit ball in $\R^2$.}

From here the proof is essentially the same as the $M=\C$ case. 
Again we obtain the upper bound on $B_1(x)$ from Corollary \ref{upper_cor}.
This time we also apply Corollary \ref{upper_cor} on the time interval $[T_0,2T_0]$
to get a bound $u_i(t)\leq M_0$ on 
$\overline{B_{7}(0)}\times [T_0,2T_0]$, and apply the Harnack inequality of Corollary \ref{lower_cor} with $\rho=3$.
This gives the required lower bound at all points $y_0\in B_\rho(x_i)$.
But $B_\rho(x_i)\supset B_{2}(0) \supset B_1(x)$.
%, and so we deduce the required lower bounds throughout $B_1(x)$, and at all times $t\in (0,T_0]$.
\end{proof}

\subsection{Extending the limit flow}

Lemma \ref{short_time_lem} gives us short time existence of a complete flow $g(t)$ starting with the measure $\mu$, when applied to the flows $g_i(t)$ constructed in Section \ref{g_i_construct_sect}.
In this section we extend this flow to the maximal existence time given in Theorem \ref{main_thm}. 

Take any sequence $s_i\downto 0$, and consider $g(s_i)$ as a smooth complete initial metric on $M$. According to Theorem \ref{main_prior_thm}, there exists a unique complete Ricci flow starting with $g(s_i)$, with an explicit maximal existence time $T_i$. 
By uniqueness of this flow, proved in \cite{ICRF_UNIQ}, while both this flow and the original short-time flow exist, they must agree. Appealing to uniqueness once more, we see that these flows give a smooth extension that does not depend on $i$.

It remains to establish that the extension flow exists for the correct time.
More precisely, we must show that $\lim_{i\to\infty} T_i$ is  
$\frac{\mu(M)}{8\pi}$ if $M=S^2$, and $\frac{\mu(M)}{4\pi}$ if $M=\C$.

What we know is that in the former case we have $T_i=\frac{\Vol_{g(s_i)}(M)}{8\pi}$, 
and in the latter case $T_i=\frac{\Vol_{g(s_i)}(M)}{4\pi}$.
Thus it remains to establish that 
$$\lim_{i\to\infty}  \Vol_{g(s_i)}(M) = \mu(M).$$
By \eqref{vol_bd} we know that $\limsup_{i\to\infty}  \Vol_{g(s_i)}(M) \leq \mu(M)$.
On the other hand, because $\mu_{g(t)}\weakto \mu$ as $t\downto 0$, we have
$\liminf_{i\to\infty}  \Vol_{g(s_i)}(M) \geq \mu(M)$.

This completes the proof of Theorem \ref{main_thm} except for the stronger claim away from the singular part. %in the case that $\mu$ has no singular part.

\subsection{Better convergence away from the singular part}

%\cmt{Warning: we don't yet know how to get $L^1$ convergence for an arbitrary solution starting at $\mu$. We only know how to do it for the solution that we construct. This is because we need to use that our particular approximators don't just converge as measures, but in $L^1_{loc}$ because we get them from mollifying. Otherwise we'd end up with uniqueness for nonsingular measures, but at the moment we don't even have uniqueness for smooth measures (because the initial data is only attained weakly).}

In this section we establish the improved attainment of the initial data on the complement of the support of the singular part of $\mu$ that is claimed in Theorem \ref{main_thm}. Multiple applications of the following proposition will imply what we claimed there.

%\cmt{The presentation below is not ideal because I am making a proposition whose proof appeals to the flows $g_i(t)$ used to construct $g(t)$. 
%I think it's OK because the proof of Theorem \ref{main_thm} is essentially spread across multiple sections}
\begin{proposition}
\label{init_prop}
Suppose that $\mu$ is a nonatomic Radon measure on a Riemann surface $M$, 
and $g(t)$ is the smooth complete conformal Ricci flow on a Riemann surface $M$,
for $t\in (0,T)$, starting with $\mu$, that has been constructed as a limit of flows $g_i(t)$ from Section \ref{g_i_construct_sect}.

Then whenever we take a coordinate chart $B_3$ on which $\mu$ is nonsingular,
and write 
$\mu=u_0\mu_0$
%$\mu=\mu_0\llcorner u_0$ \cmt{NOTATION} 
for $\mu_0$ the Lebesgue measure and write $g(t)=u(t)(dx^2+dy^2)$, we have
$$u(t)\to u_0\quad\text{ in }L^1(B)$$
as $t\downto 0$.
\end{proposition}
We will prove the proposition using the following relative of Lemma \ref{even_stronger_again_lemma}.
One can compare with results from the literature of the fast diffusion equation, e.g. \cite{RVE97}.

\begin{lemma}
\label{loc_L1_growth_lem}
Suppose $g_i(t)=u_i(t)(dx^2+dy^2)$, $i=1,2$, are Ricci flows on the disc $B_2$ for $t\in [0,T)$, and suppose that $u_1(t)\geq \ep t$, for some $\ep>0$, throughout
$B_2\times [0,T)$. Then 
$$\left(\int_{B} [u_2(t)-u_1(t)]_+\right)^\half\leq
\left(\int_{B_2} [u_2(0)-u_1(0)]_+\right)^\half + c\left(\frac{t}{\ep}\right)^\half,$$
for all $t\in [0,T)$, where $c$ is universal.
\end{lemma}

%\cmt{I thought this argument was clearer than the iteration. The result is a little sharper also. OK?}

\begin{proof}[Proof of Lemma \ref{loc_L1_growth_lem}]
We start by choosing a cut-off function $\vph\in C_c^\infty(B_2)$, taking values in $[0,1]$, with the properties that $|\lap\vph|\leq \be \vph^\half$ for some constant $\be$, and $\vph\equiv 1$ on $B$. Then we can compute
\begin{equation}\begin{aligned}
\frac{d}{dt}\int_{B_2}[u_2-u_1]_+\vph &= \int_{\{u_2>u_1\}}[\lap \log u_2-\lap\log u_1]\vph\\
&\leq \int_{B_2} [\log u_2-\log u_1]_+ |\lap\vph|\\
&\leq \be\int_{B_2}  2\left[\frac{u_2}{u_1}-1\right]_+^\half \vph^\half\\
&\leq 4\be\left(\frac{\pi}{\ep t}\right)^\half \left(\int_{B_2} [u_2-u_1]_+ \vph\right)^\half
\end{aligned}\end{equation}
where we used the inequality $\log s\leq 2[s-1]^\half$, for $s\geq 1$, 
and Cauchy-Schwarz.
Further details of this type of calculation can be found in \cite[Page 38]{giesen_thesis}.
Therefore
$$\frac{d}{dt}\left(\int_{B_2}[u_2-u_1]_+\vph\right)^\half
\leq 2\be \sqrt{\frac{\pi}{\ep}}t^{-\half},$$
which integrates to give the lemma.
\end{proof}

\begin{proof}[Proof of Proposition \ref{init_prop}]
%show that if the initial measure $\mu$ has no singular part then the subsequent Ricci flow $g(t)$ that we have already constructed attains its initial data in the $L^1_{loc}$ sense.
The  Ricci flow $g(t)=u(t)(dx^2+dy^2)$ for $t\in (0,T]$  arises as a smooth local limit on $B_3\times (0,T]$ of flows $g_i(t)=u_i(t)(dx^2+dy^2)$
for $t\in [0,T]$. 
We  know that 
$$u_i(t)\geq \ep t \quad \text{and hence}\quad u(t)\geq \ep t,$$
on $B_2$ for some $\ep>0$, after possibly reducing $T>0$.

%\cmt{Pedantry - we reduced $B_3$ to $B_2$ in case this new chart $B_3$ was not one we considered before. It could be degenerate at the boundary... But $B_2$ is anyway what we need.}

By the assumption that $\mu$ has no singular part, Lemma \ref{smooth_approx_lem} tells us that 
%we can assume now that there exists $u_0\in L^1_{loc}(B_3)$ such that 
$u_i(0)\to u_0$ in $L^1_{loc}(B_3)$ as $i\to\infty$.
However, the initial data of $g(t)$ is only initially known to be attained weakly in the sense that for all $\vph\in C_c^0(B_3)$ we have 
$$\int_{B_3} u(t)\vph \to \int_{B_3} u_0\vph$$
as $t\downto 0$.
We would like to show that we have the strong convergence
$$u(t)\to u_0\text{ in }L^1(B).$$
In particular, for arbitrarily small $\eta>0$, we need to show that 
$$\|u(t)-u_0\|_{L^1(B)}<\eta$$
for sufficiently small $t>0$.

By the $L^1_{loc}(B_3)$ convergence $u_i(0)\to u_0$, we can \emph{fix} $i\in\N$ such that for all $j\geq i$ (in particular for $j=i$) we have
$$\|u_j(0)-u_0\|_{L^1(B_2)}\leq \eta/8.$$
Thus, for $j\geq i$ we have
$$\|u_i(0)-u_j(0)\|_{L^1(B_2)}\leq \eta/4.$$
By Lemma \ref{loc_L1_growth_lem}, we know that 
$$\int_{B} |u_i(t)-u_j(t)|\leq
2\int_{B_2} |u_i(0)-u_j(0)| + C_0t 
\leq \eta/2 + C_0t.$$
Sending $j\to\infty$ gives
$$\|u_i(t)-u(t)\|_{L^1(B)}\leq \eta/2 + C_0t.$$
We now insist that $t>0$ is sufficiently small so that $C_0t<\eta/8$, 
and also so that $\|u_i(t)-u_i(0)\|_{L^1(B)}\leq \eta/8$ for our fixed $i$.
By the triangle inequality we can then conclude
\begin{equation}\begin{aligned}
\|u(t)-u_0\|_{L^1(B)}&\leq
\|u(t)-u_i(t)\|_{L^1(B)}+\|u_i(t)-u_i(0)\|_{L^1(B)}
+\|u_i(0)-u_0\|_{L^1(B)}\\
&<(\eta/2+\eta/8)+\eta/8+\eta/8\\
&<\eta,
\end{aligned}\end{equation}
as required to complete the proof.
\end{proof}

\noindent
At this point the proof of Theorem \ref{main_thm} is complete.

\medskip

Before leaving this section we remark that although Lemma \ref{short_time_lem} is stated for Ricci flows $g_i(t)$ that are smooth down to $t=0$, a minor adjustment of the proof 
gives the following variation with measure-valued initial data.
\begin{lemma}
\label{short_time_lem_updated}
Suppose that we have a sequence of smooth complete conformal Ricci flows $g_i(t)$ on a Riemann surface $M$, for $t\in (0,T)$, such that the volume measures of each $g_i(t)$ converge weakly to nontrivial nonatomic measures $\mu_i$ as $t\downto 0$,
and these initial measures $\mu_i$ converge weakly to a nontrivial
nonatomic Radon measure $\mu$ as $i\to\infty$.
Then after passing to a subsequence, and possibly reducing $T>0$, we have
$g_i(t)\to g(t)$ smoothly locally on $M\times (0,T)$, where $g(t)$ is itself a smooth complete Ricci flow on $M$ for $t\in (0,T)$. 
Moreover, the volume measure of $g(t)$ converges weakly to $\mu$ as $t\downto 0$.
\end{lemma}

%%%%%%%%%%%%%%%%%%%%%%%%%%%%%%%%%%%%%%%%%%%%%%%%%%

%\newpage

\section{Nongradient Ricci solitons originating from measure data}
\label{soliton_sect}

In \cite{breather} it was shown how to use appropriate periodically scale-invariant smooth initial data in Theorem \ref{main_prior_thm} in order to construct nontrivial breathers. Now that we can start the Ricci flow with rougher initial data, we can do a related construction in order to construct new examples of Ricci solitons, answering a number of open problems. 
For example, we show how to construct
nongradient Ricci solitons on a surface,
and solitons on $\R^2$ without rotational symmetry.

%\cmt{In the process we solve both Problems 1.86 and 1.87 from 
%\cite[Page 51]{chowIIgeometric}.} 

\subsection{Ricci flow starting with a line: An explicit example}
\label{explicit_ex_sect}

One of the simplest nontrivial instances in which we can apply our existence theorem \ref{main_thm} is when the underlying Riemann surface is $\C\simeq \R^2$ and the starting measure is the uniform measure on a line. More precisely, define $L:=\{0\}\times \R \subset\R^2$ to be the $y$-axis in the plane, and set
\begin{equation}
\label{soliton_initial_measure}
\mu=\h^1\llcorner L.
\end{equation}
This measure is invariant under translations in the $y$ direction, and when we pull it back by a dilation $\vph_\la(\bx)=\la \bx$, for $\la>0$, the measure is scaled by a factor $\la$.

%{\color{red} OLD: 
%Both translations and dilations are biholomorphic maps, so if we believe in our uniqueness conjecture \ref{uniqueness_conj} then these invariance properties will imply invariance properties for the Ricci flow $g(t)$ on $\R^2$, for $t> 0$, starting at $\mu$, whose existence is given by Theorem \ref{main_thm}.
%In particular, the pull-back $\vph_\la^* g(t)$ and the scaled flow $\la g(t/\la)$ will both be complete Ricci flows that start with the same measure $\la\mu$, so
%$$\vph_\la^* g(t)=\la g(t/\la).$$
%In particular, by first setting $t=1$, and then redefining $t=1/\la$, we find that
%$$g(t)=t\vph_{\frac{1}{t}}^* g(1),$$
%and we have an expanding Ricci soliton. 
%
%This argument is not complete as written because we have not established uniqueness in this generality. However, it turns out that we have enough information from the discussion above to explicitly construct this soliton. In order to keep the constants as clean as possible we opt to analyse the soliton generated by flowing the initial measure in \eqref{soliton_initial_measure} scaled by a factor of $2\pi$.}

Both translations and dilations are biholomorphic maps. Assuming uniqueness as in 
Conjecture \ref{uniqueness_conj} is enough to turn these invariance properties into invariance properties for the Ricci flow $g(t)$ on $\R^2$, for $t> 0$, starting at $\mu$, whose existence is given by Theorem \ref{main_thm}.
In particular, the pull-back $\vph_\la^* g(t)$ and the scaled flow $\la g(t/\la)$ are both complete Ricci flows that start with the same measure $\la\mu$, so having uniqueness gives
$$\vph_\la^* g(t)=\la g(t/\la).$$
By first setting $t=1$, and then redefining $t=1/\la$, we find that
$$g(t)=t\vph_{\frac{1}{t}}^* g(1),$$
and we have an expanding Ricci soliton. 

This argument is not complete here because we are not establishing uniqueness in this 
generality in this paper.\footnote{though see the later developments of \cite{TY4} and \cite{PT} that were made since this paper was first released.}
However, it turns out that we have enough information from the discussion above to explicitly construct this soliton. In order to keep the constants as clean as possible we opt to analyse the soliton generated by flowing the initial measure in \eqref{soliton_initial_measure} scaled by a factor of $2\pi$.

\begin{theorem}
On $\R^2$, the complete Riemannian metric $g_0=u_0(x,y)(dx^2+dy^2)$ defined by 
$$u_0(x,y):=\frac{2}{1+x^2}$$
is a nongradient expanding Ricci soliton.
In particular, 
$$g(t)=t\vph_{\frac{1}{t}}^* g_0$$
is a complete Ricci flow for $t>0$ that can be written explicitly as 
$g(t)=u(x,y,t)(dx^2+dy^2)$ where
\begin{equation}
\label{RF_explicit}
u(x,y,t)=\frac{2t}{t^2+x^2}
\end{equation}
and satisfies
\begin{equation}
\mu_{g(t)}\weakto 2\pi \h^1\llcorner L.
\end{equation}
\end{theorem}
The proof is a straightforward computation. In particular, we can verify that the 
conformal factor that is given explicitly by \eqref{RF_explicit} satisfies
$$\pl{u}{t}=\frac{2(x^2-t^2)}{(t^2+x^2)^2}=\lap\log u.$$
The assertion that the soliton is nongradient follows from the well-known fact that a gradient Ricci soliton on a surface with potential function $f$ must have the rotated gradient $J\grad f$ being a Killing field (see e.g. \cite[Chapter 1, \S 3.1]{chowIIgeometric}). 
%e.g. Ben Chow p.11 of "Part 1: Geometric aspects".
Up to scaling, the only Killing field for $g_0$ is vertical translation.
This is the first nongradient soliton that is known in two dimensions. It extends to higher dimensions by taking a product with a Gaussian soliton.
In higher dimensions, homogeneous nongradient solitons have been constructed in
\cite{baird, lott}.
In contrast, gradient solitons have been classified in two dimensions; see e.g. \cite{Bernstein_Mettler}. 

We see from the formula for this soliton that this Ricci flow diffuses the mass over length scales of order $t$. This is consistent with the estimates elsewhere in this paper, and can be contrasted with the standard linear heat equation, which averages over length scales of order $\sqrt{t}$.

As for the geometry of the metric $g_0$, we can compute its Gauss curvature to be 
$$K=\half\left(\frac{1-x^2}{1+x^2}\right).$$
For large $x$, it looks like spatial infinity in the half-space model of hyperbolic space (scaled to have Gauss curvature $-\half$).
Along the $y$ axis the Gauss curvature is $\half$.
The effect of the Ricci flow away from the $y$-axis begins like the big-bang Ricci flows on the half-spaces on either side of the $y$ axis. 
The Gauss curvature is controlled by
$$-\frac{1}{2t} < K_{g(t)}\leq \frac{1}{2t}.$$

\begin{remark}
Because the flow is independent of $y$, we can quotient it by a translation $y\mapsto y+c$ or  consider it as a one-dimensional solution of the equation, depending only on $x$. In the latter case the solution arises as a limit of Barenblatt-Pattle solutions 
\cite{DDD}. However, both of these viewpoints destroy the soliton structure. 
\end{remark}

%\cmt{Side point: If comparing with the twice punctured spherical metric, then on the cylinder it would have conformal factor $\frac{1}{\cosh^2 x}$, I think}

\begin{remark}
The metric we construct here can be considered as being dual to the hyperbolic plane, scaled to have curvature $-\half$, in the sense of Buscher duality from string theory.
See the discussion in \cite{chowIIgeometric}.
\end{remark}

%\cmt{
%Let's analyse the metric $g_0$ a little more. Introduce a coordinate $s$ in place of $x$, which is the geodesic distance from the $y$-axis. Thus
%$$s=\int_0^x \left(\frac{2}{1+x^2}\right)^\half dx.$$
%We can solve to give $x=\sinh\frac{s}{\sqrt{2}}$.
%We can then write 
%$$g_0=ds^2+w^2(s)dy^2,$$
%where 
%$$w(s)=\frac{\sqrt{2}}{\cosh\frac{s}{\sqrt{2}}}.$$
%Note that the formula for the Gauss curvature of such a metric is 
%$$K=-\frac{w''}{w},$$
%so in this case the curvature is $\half\frac{1-x^2}{1+x^2}$ as expected.
%
%The Buscher dual of $g_0$ is $ds^2+w^{-2}(s)dy^2$, i.e. we replace the warping factor by
%$$\tilde w(s)=\frac{\cosh\frac{s}{\sqrt{2}}}{\sqrt{2}}.$$
%We can then compute the curvature to be $-\half$. Since we are working on a simply connected space, we must be considering the (scaled) hyperbolic plane.
%}

%\begin{rem}

\subsection{More Ricci solitons on the plane}

The idea  motivating the construction above generalises to give a large class of additional examples of Ricci solitons on surfaces. By lifting them to the universal cover, we can consider them on $\R^2$ and the half plane.

\subsubsection{Translating solitons in the plane}
\label{translate_sect}

If we are given a nonnegative function $f\in L^1_{loc}(\R)$ that is not identically zero, we can generate a function $u_0:\R^2\to [0,\infty)$ 
%that is not identically zero 
by defining
$$u_0(x,y):=e^{x}f(y).$$
Taking $u_0$ to be the density of a measure $\mu$ on the plane, we can apply Theorem \ref{main_thm} to give a Ricci flow $g(t)$ starting with $\mu$.
Defining the 
shift $T_\si:\R^2\to\R^2$ by $T_\si(x,y)=(x+\si,y)$, we see that
$$T_\si^*\mu = e^{\si}\mu,$$
and so both $T_\si^*g(t)$ and $e^{\si}g(t e^{-\si})$ are Ricci flows starting with the same initial measure. By the uniqueness of Theorem \ref{L1_uniq} we deduce that
$$e^{\si}g(t e^{-\si})=T_\si^*g(t),$$ 
which implies
$$g(t)=t\, T^*_{-\log t}g(1),$$
and we have generated an expanding soliton on $\R^2$ that moves by translation.

For context, consider the trivial case that $f\equiv 1$. 
Then geometrically $\mu$ is the volume measure of the universal cover of the punctured plane. The soliton we have described in this trivial case is the universal cover of the time $1$ instantaneously complete evolution of the punctured plane.

%\cmt{can we find any explicit solutions for specific $f\in L_{loc}^1(\R)$?}

\subsubsection{Dilating and rotating solitons in the plane}
\label{dilate_sect}

As a variation of the  construction above, 
if we are given $\al>0$, $\be\in\R$ and a nonnegative function $f\in L^1_{loc}(S^1)$ that is not identically zero, we can generate a function $u_0:\R\times S^1\to [0,\infty)$ 
by defining
$$u_0(x,\theta):=e^{\al x}f(\theta+\be x).$$
Taking $u_0$ to be the density of a measure on the cylinder $\R\times S^1$, and  pushing it forward to the plane under the conformal map $(x,\theta)\mapsto (e^x,\theta)$ gives a Radon measure $\mu$ on the whole plane with density relative to Lebesgue measure  given in polar coordinates by the $L^1_{loc}$ function
$$(r,\theta)\mapsto r^{\al-2}f(\theta+\be\log r).$$ 
Note that asking for $\al$ to be strictly positive makes $\mu$ locally finite, in particular near the origin. 
%In fact, by construction the measure $\mu$ can be written as Lebesgue measure weighted by a $L^1_{loc}$ function.
That density function has been constructed so that if we pull back $\mu$ under the conformal diffeomorphism 
$\psi_\la:\R^2\to\R^2$ defined in polar coordinates by 
$$\psi_\la(r,\theta)=(\la r, \theta-\be\log\la),$$
%\cmt{Please check the sign here!}
where $\la>1$
(corresponding to adding $\log \la$ to $x$ and rotating by $-\be\log\la$ in the cylinder picture)
then it is scaled according to
$$\psi_\la^*\mu=\la^\al \mu.$$
A trivial example would be $f\equiv 1$, $\be=0$ and $\al=2$, which gives the flat plane. Changing $\al$ adjusts this to a nontrivial cone metric. 

Let $g(t)$ be the Ricci flow starting with $\mu$, as given by Theorem \ref{main_thm}.
Consider starting the Ricci flow with the initial measure $\la^\al\mu$.
On the one hand, $\psi_\la^*g(t)$ is such a flow, and on the other hand
$\la^\al g(t/\la^\al)$ is such a flow. 
By the uniqueness of Theorem \ref{L1_uniq}, we then must have
$$\psi_\la^*g(t)=\la^\al g(t/\la^\al).$$
Then 
$$g(t)=t \psi_{t^{-1/\al}}^* g(1),$$
and we have an expanding Ricci soliton.
In the case that $\be\neq 0$, we see that the diffeomorphism $\psi_{t^{-1/\al}}$ rotates an unbounded number of times as $t\downto 0$.

A simple example corresponding to the case $\al=2$ is the soliton generated by the measure on the plane that is Lebesgue measure in a half space, and the trivial measure on the complementary half space.

\subsection{Solitons in the half plane}

Essentially identical constructions to the ones above give a large class of new expanding solitons on the half plane, or equivalently on the disc. 

First, we can modify Section \ref{translate_sect} by taking instead 
$f\in L^1_{loc}(0,\infty)$, in which case $u_0$ will live on $\R\times (0,\infty)$. 
We then apply Theorem \ref{main_thm} on the half plane instead, and the rest of the argument follows verbatim with the shift $T_\si$ restricted to this half plane.

Meanwhile, we have a half plane version of the dilating solitons of Section 
\ref{dilate_sect} arising by restricting $f$ to half a circle, i.e. the interval $(0,\pi)$, and setting $\be=0$. 
Thus the density is given in polar coordinates by
$$(r,\theta)\mapsto r^{\al-2}f(\theta),$$ 
We apply Theorem \ref{main_thm} in the half plane and end up with an expanding soliton
of the form
$$g(t)=t \vph_{t^{-1/\al}}^* g(1).,$$

\begin{ex}
\label{half_space_ex}
A simple example corresponding to the case $\al=2$ and $f\equiv 1$ is the soliton generated by the half plane $\{y>0\}$ with the Euclidean metric.
The corresponding Ricci flow could be written $u(x,y,t)(dx^2+dy^2)$ where
$$u(x,y,t)=F\left(\frac{y}{\sqrt{t}}\right)$$
for some decreasing function  $F:(0,\infty)\to (1,\infty)$ satisfying an appropriate ODE with asymptotic behaviour 
$F(s)\downto 1$ as $s\to\infty$ (which is guaranteed by the smooth local convergence of any Ricci flow generated by Theorem \ref{main_prior_thm} to its smooth initial data)
and $F(s)\geq\frac{2}{s^2}$ (which is forced because the Ricci flow corresponding to the conformal factor $F\left(\frac{y}{\sqrt{t}}\right)$ must lie above the conformal factor of the big-bang Ricci flow corresponding to the conformal factor $\frac{2t}{y^2}$.)
One could check that $F(s)\sim\frac{2}{s^2}$ for small $s>0$.
\end{ex}

%\cmt{Can we write down explicitly the example above?}

\subsection{Solitons arising from singular measures}

%{\color{red} OLD:
%
%The constructions above can be generalised further, replacing each of the $L^1_{loc}$ functions $f$ by more general Radon measures. If we could assume the uniqueness conjecture \ref{uniqueness_conj} then the constructions would be identical to the ones above. 
%For example, the explicit example of Section \ref{explicit_ex_sect} would arise from the abstract construction of Section \ref{dilate_sect} by taking $f$ to be the sum of two delta masses on opposite sides of the circle, and setting $\al=1$ and $\be=0$.
%The interested reader can verify that we can nevertheless treat these cases rigorously by mollifying the Radon measures that substitute for $f$, generating the corresponding expanding Ricci solitons as above, and then taking a limit as we mollify less and less, using the estimates developed in this paper.
%
%As an example, one could imagine the soliton generated by flowing the Hausdorff measure $\h^1$ restricted to a logarithmic spiral in the plane.
%
%}

The constructions above can be generalised further, replacing each of the $L^1_{loc}$ functions $f$ by more general Radon measures. 
For example, the explicit example of Section \ref{explicit_ex_sect} would arise from the abstract construction of Section \ref{dilate_sect} by taking $f$ to be the sum of two delta masses on opposite sides of the circle, and setting $\al=1$ and $\be=0$.
We do not pursue this matter here because the optimal route is via the uniqueness conjecture \ref{uniqueness_conj}, in which case the constructions would be identical to the ones above. 
Nevertheless, the interested reader can verify that even without Conjecture \ref{uniqueness_conj} we can treat these cases rigorously here by mollifying the Radon measures that substitute for $f$, generating the corresponding expanding Ricci solitons as above, and then taking a limit as we mollify less and less, using the estimates developed in this paper.

As an example, one could imagine the soliton generated by flowing the Hausdorff measure $\h^1$ restricted to a logarithmic spiral in the plane.

%%%%%%%%%%%%%%%%%%%%%%%%%%%%%%%%%%%%%%%%%%%%%%%%%%

\section{Smooth Ricci flows attaining smooth initial data weakly}
\label{DSS_Q}

In this section we prove Theorem \ref{counter_ex_thm}, asserting the existence of a Ricci flow starting with the Euclidean plane in the weak sense that the Riemannian distance converges, but that is not the normal static solution. 
\begin{proof}
The Ricci flow $g(t)$ will arise directly as in the proof of Theorem \ref{main_thm}, except that we would like to specify the approximations $g_i$ explicitly rather than by appealing to Lemma \ref{smooth_approx_lem}. (It is not significant that we will allow our approximating metrics to have infinite volume.)
First we ask that $g_i=u_i(x)(dx^2+dy^2)$ for some function $u_i:\R\to (0,\infty)$
with $u_i(x)=u_i(-x)$. 
That is, the metric is invariant under vertical translations and by reflection in the $y$ axis.

Next we ask that $g_i\geq g_0$ for all $i$. This bound will be inherited by the subsequent flows $g_i(t)$ given by Theorem \ref{main_prior_thm}, i.e. $g_i(t)\geq g_0$ for all $t\geq 0$, and ultimately inherited by the limit Ricci flow $g(t)=u(x,t)(dx^2+dy^2)$. Consequently we have  $u(x,t)\geq 1$, and thus
$$d_{g(t)}\geq d_{g_0}$$
for all $t>0$ and throughout $\R^2\times \R^2$.

It will be even more helpful to ask that $g_i=g_0$ on $\{|x|\geq 1/i\}$, i.e. away from a strip around the $y$ axis, since this will allow us to derive  
an upper barrier for $g_i$, which in turn will give \emph{upper} bounds on $d_{g(t)}$.
Recall the soliton proposed in Example \ref{half_space_ex}
that originates from flowing a half space with the Euclidean metric.
As discussed there, if we consider that soliton on the right half plane $\{x>0\}$, then we can write it as $F\left(\frac{x}{\sqrt{t}}\right)(dx^2+dy^2)$.
(Recall that $F(s)$ is a decreasing function that converges to $1$ as $s\to\infty$.)
Because $F\geq 1$, the Ricci flow on the shifted half plane $\{x>1/i\}$ with conformal factor $F\left(\frac{x-1/i}{\sqrt{t}}\right)$ will (being `maximally stretched' \cite{GT2}) lie above the conformal factor $u_i(x,t)$ of the restriction of $g_i(t)$ to 
$\{x>1/i\}$, i.e.
$$u_i(x,t)\leq F\left(\frac{x-1/i}{\sqrt{t}}\right),$$
for $x>1/i$ and $t>0$.
%If we ask that $F\left(x\sqrt{i}\right)$ is an upper barrier for $g_i$, then the `maximally stretched' Ricci flow with conformal factor $F\left(\frac{x}{\sqrt{\frac{1}{i}+t}}\right)$ will lie above $g_i(t)$. 
If we then extract a (subsequential) limit Ricci flow $g(t)=u(x,t)(dx^2+dy^2)$, as in the proof of Theorem \ref{main_thm}, it will inherit the upper bound
$$u(x,t)\leq F\left(\frac{x}{\sqrt{t}}\right).$$
We can use this upper bound in order to control $u(x,t)$ to be close to $1$ outside a strip around the $y$ axis $L$ that is becoming thinner and thinner. More precisely, given arbitrarily small $\ep>0$, 
pick $Q>0$ large enough so that $F(Q)\leq (1+\ep)^2$. 
Then for $x\geq Q\sqrt{t}$ we have
$$u(x,t)\leq F\left(\frac{x}{\sqrt{t}}\right)\leq F(Q)\leq(1+\ep)^2.$$
Although $Q$ is liable to be large, we will be able to take such small $t$ that 
$Q\sqrt{t}$ can be made as small as we like, i.e. the strip can be made very thin.

This upper bound for $u$ gives us very good control on $d_{g(t)}(p,q)$ for $p$ and $q$ in the half plane $\{x\geq Q\sqrt{t}\}$. Precisely, we have
$$d_{g(t)}(p,q)\leq (1+\ep)d_{g_0}(p,q).$$
By symmetry, we also then have control for $p$ and $q$ in the opposite half plane $\{x\leq -Q\sqrt{t}\}$.
All remaining cases can be reduced to the problem of controlling 
$d_{g(t)}(p,q)$ for $p\in L$ (i.e. $p$ on the $y$ axis) and $q$ in the half plane
$\{x\geq 0\}$. By imagining shifting both $p$ and $q$ to the right by an amount $Q\sqrt{t}$, and using the case above, we reduce to showing that 
the $d_{g(t)}$ distance between a point on the $y$ axis $L$ and its shift to the right by $Q\sqrt{t}$ can be made as small as we like. But
$$
d_{g(t)}((0,0),(Q\sqrt{t},0)) =\int_0^{Q\sqrt{t}} \sqrt{u(x,t)}dx
\leq Q^\half t^\frac{1}{4}\left(\int_0^{Q\sqrt{t}} u(x,t)dx\right)^\half
$$
and 
$$\int_{-1}^1 u(x,t)dx \to 3$$
as $t\downto 0$ because  $\Vol_{g(t)}([-1,1]\times [0,1])$ will have to converge to $\mu([-1,1]\times [0,1])=3$.
\end{proof}
%\cmt{Above we are using that the $\mu$ measure of the boundary of the rectangle considered is zero.}

It may be worth pointing out that the length of a vertical path between two points 
on the $y$ axis $L$ that are Euclidean distance $1$ from each other will blow up as $t\downto 0$. In reality, the vertical  path is not the shortest.

We also observe that the Ricci flow that we construct for Theorem \ref{counter_ex_thm} has the additional property that it converges smoothly locally to the Euclidean plane on 
$\R^2\setminus L$.

\begin{remark}
Following on from this paper, M.-C. Lee and the first author have refined 
the properties of the Ricci flow introduced in Theorem \ref{counter_ex_thm}; see
\cite{LT1}.
\end{remark}

\section{There cannot exist a solution starting with a Dirac mass}

In this section we prove Theorem \ref{no_dirac_mass_thm}, ruling out the existence of flows starting at  measures with isolated atoms.
An important ingredient will be the following a priori estimate that controls an arbitrary Ricci flow  starting locally at initial data with bounded conformal factor.

%\cmt{changed so $u=0$ is OK. Needs checking.}

\begin{lemma}
\label{local_upper_bds_bdd_init_data_lem}
Suppose $g(t)=u(t)(dx^2+dy^2)$ is a smooth Ricci flow on $B_R$, some $R>0$, for $t\in (0,T)$.
Suppose that $u_0\in L^\infty(B_R)$ with $0\leq u_0\leq L$.
Suppose $\mu_{g(t)}$ converges weakly to the measure 
$\mu:=u_0\mu_0$ 
%$\mu:=\mu_0\llcorner u_0$ \cmt{NOTATION}
corresponding to $u_0$ as $t\downto 0$, where $\mu_0$ is the Lebesgue measure.
Then for $t\in (0,T)$ satisfying $t\leq R^2L$ the conformal factor at the origin is controlled by
$$u(0,t)\leq C_1 L,$$
%on the ball centred at $0$,  of radius $R-(\frac{2 t}{L})^\half$, 
for universal $C_1$.
%where $C_0$ is the universal constant from Theorem  \ref{TY1_thm}.
%In particular, we have the upper bound $u(0,t)\leq C_0L/2$ for times 
%$t\in (0,R^2L/2)$ %. The lemma is vacuous at later times.}
\end{lemma}

\begin{proof}
Because $t\leq R^2L$, if we set $r= \frac{1}{3}(\frac{t}{L})^\half$ then $B_{3r}\subset B_R$.
Because $u_0\leq L$, we must have $\mu(B_{3r})\leq L\pi(3r)^2=\pi t$.
By the weak convergence $\mu_{g(t)}\weakto \mu$
as $t\downto 0$, for sufficiently small $t_0\in (0,t)$ we must have $\Vol_{g(t_0)}(B_{2r})\leq 2\pi (t-t_0)$. 
Thus we can apply Theorem \ref{TY1_thm}, thinking of time $t_0$ here as time $0$ there,
to deduce that 
$$u(0,t)\leq C_0r^{-2}(t-t_0)=9C_0L(t-t_0)/t\leq 9C_0 L.$$
\end{proof}

%\cmt{There is a more complicated result saying that a lower bound gives (some other) lower bound for the flow, but that requires some global information. Still, it may be useful to control the potential when studying uniqueness.}

\begin{proof}[Proof of Theorem \ref{no_dirac_mass_thm}]
Suppose a flow $g(t)$ did exist, contrary to the claim of the theorem. 
By taking an appropriate coordinate chart $B_2$ in $M$ we would have a Ricci flow 
$g(t)=u(t)(dx^2+dy^2)$ for $t\in (0,\ep)$ and  a smooth function $u_0:B\to [0,1]$
such that %\cmt{NOTATION times 2}
$$u(t)\mu_0 \weakto \de_0 + u_0\mu_0$$
as $t\downto 0$.
Let $C_1$ be the universal constant from Lemma \ref{local_upper_bds_bdd_init_data_lem}. 
Then according to Lemma \ref{local_upper_bds_bdd_init_data_lem}, for each $x\in B\setminus\{0\}$, we have
$$u(x,t)\leq C_1 $$
for all $t\in (0,T)$ with $t\leq |x|^2$, or equivalently we have the bound for all $t\in (0,T)$ and all $x\in B\setminus B_{\sqrt{t}}$.
Therefore for each $r\in (0,\half)$, if $u_r(t)$ is the conformal factor of the unique smooth instantaneously complete Ricci flow on $B\setminus B_r$ starting with a conformal factor identically equal to $C_1$
then $u(t)\leq u_r(t)$ where defined, for all $t\in (0,T)$.
Here we start the comparison at some time $t<|x|^2$.
%\cmt{Strictly speaking we start the comparison at an arbitrarily small time so that $u$ is smooth down to the initial time...}
Letting $r\downto 0$, we find that $u(t)$ lies below the conformal factor of 
the unique smooth instantaneously complete Ricci flow on $B\setminus \{0\}$ starting identically equal to $C_1$.
In particular, for each $r>0$ (however small) we have
$$\limsup_{t\downto 0}\Vol_{g(t)}(B_r)\leq C_1\pi r^2,$$
which converges to zero as $r\downto 0$.
This contradicts the weak convergence to $\mu$ since there cannot then be a delta mass at the origin.
\end{proof}

%%      ---------------------------------------------------------------------
%%      --------------------------- BIBLIOGRAPHY ----------------------------
%%      ---------------------------------------------------------------------
%% PUT HERE THE BIBLIOGRAPHY IN YOUR FAVOURITE FORMAT
%% Please check that the format of the bibliography is uniform and coherent

\end{document}